\documentclass[11pt]{article}
\usepackage{standalone,tikz}
\usepackage{graphics,verbatim}
\usepackage{amssymb,amsmath,amsthm}
\usepackage[font=footnotesize]{caption}
\usepackage{fullpage}
\usepackage{setspace}
\newtheorem{theorem}{Theorem}[section]

\newtheorem{corollary}[theorem]{Corollary}
\newtheorem{conjecture}[theorem]{Conjecture}
\newtheorem{question}[theorem]{Question}
\newtheorem{lemma}[theorem]{Lemma}

\newtheorem{proposition}[theorem]{Proposition}
\newtheorem{claim}{Claim}

\newtheorem{claimab}{Claim}

\newtheorem{claimI}{Claim}

\theoremstyle{definition}
\newtheorem{definition}[theorem]{Definition}

\def\q{\hfill\rule{1ex}{1ex}}

\RequirePackage{marginnote,hyperref}
\begin{document}

\title{Circular Coloring and Fractional Coloring in Planar Graphs}

\author{Xiaolan Hu\thanks{School of Mathematics and Statistics \& Hubei Key Laboratory of Mathematical Sciences,
Central China Normal University, Wuhan 430079, People's Republic of China. Email: xlhu@mail.ccnu.edu.cn. Research partially supported by  National Natural Science Foundation of China (No. 11971196).} ~~and~~ Jiaao Li\thanks{School of Mathematical Sciences and LPMC, Nankai University, Tianjin 300071, People's Republic of China. Email: lijiaao@nankai.edu.cn. Research partially supported by  National Natural Science Foundation of China (No. 11901318) and Natural Science Foundation of Tianjin (No. 19JCQNJC14100).}}
\date{}
\maketitle

\begin{abstract}
We study the following Steinberg-type problem on circular coloring: for an odd integer $k\ge 3$, what is the smallest number $f(k)$ such that every planar graph of girth $k$ without cycles of length from $k+1$ to $f(k)$  admits a homomorphism to the  odd cycle $C_k$ (or equivalently, is circular $(k,\frac{k-1}{2})$-colorable).
Known results and counterexamples on Steinberg's Conjecture indicate that $f(3)\in\{6,7\}$. In this paper,  we show that $f(k)$ exists if and only if $k$ is an odd prime. Moreover, we prove that for any prime $p\ge 5$, $$p^2-\frac{5}{2}p+\frac{3}{2}\le f(p)\le 2p^2+2p-5.$$
We conjecture that $f(p)\le p^2-2p$, and  observe that the truth of this conjecture implies Jaeger's conjecture that every  planar graph of girth $2p-2$ has a homomorphism to  $C_p$ for any prime $p\ge 5$. Supporting this conjecture, we prove a related fractional coloring result  that every planar graph of girth $k$ without cycles of length from $k+1$ to $\lfloor\frac{22k}{3}\rfloor$ is fractional $(k:\frac{k-1}{2})$-colorable for any odd integer $k\ge 5$.
\\[2mm]
\end{abstract}

\section{Introduction}
 The circular chromatic number of a graph is a natural generalization of the chromatic number
of a graph, introduced by Vince \cite{Vince}.
For two positive integers $k$ and $d$ with $k\geq 2d$, a {\em circular $(k,d)$-coloring} of a graph $G$ is a mapping $\varphi: V(G)\rightarrow \{0,1,\ldots,k-1\}$ such that $d\leq |\varphi(u)-\varphi(v)|\leq k -d$ whenever $uv\in E(G)$. The circular chromatic number $\chi_c(G)$ of $G$ is defined as the infimum of rational
numbers $\frac{k}{d}$ for which $G$ has a circular $(k,d)$-coloring. Notice that a circular $(k,1)$-coloring of a graph $G$ is just an ordinary proper $k$-coloring of $G$. We call $\chi_c(G)$ a refined measure of coloring because $\chi(G)-1< \chi_c(G)\leq \chi(G)$ for every
graph $G$, as proved in \cite{BonHell,Vince}, where $\chi(G)$ is the chromatic number of $G$.
Perhaps one of the most intriguing  problems concerning circular coloring  of planar graphs is the following conjecture, motivated from the dual of Jaeger's circular flow conjecture \cite{Jaeger}.

\begin{conjecture}\label{CONJ:circular-conj}
For any planar graph $G$ of girth at least $2t$, $\chi_c(G)\le 2+\frac{2}{t}$.
\end{conjecture}
The $t=1$ case of this conjecture is the celebrated Four Color Theorem proved by Appel and Haken \cite{AH76} in 1976; the $t=2$ case is the classical Gr\"otzsch's theorem \cite{Grotzsch} from 1959 that every triangle-free planar graph is $3$-colorable. Conjecture \ref{CONJ:circular-conj} remains  open for each $t\ge 3$. A result of Hell and Zhu \cite{HZ00} shows that Conjecture \ref{CONJ:circular-conj} is true for $K_4$-minor-free graphs, a subclass of planar graphs.

When $t=2s$ is even, it is not hard to observe that a graph $G$ is circular $({2t+2},{t})$-colorable if and only if  $G$ admits a homomorphism to the odd cycle $C_{2s+1}$. Indeed, $\chi_c(C_{2s+1})=\frac{2s+1}{s}$ and each color class appears in exactly one  vertex of $C_{2s+1}$ under a circular $({2s+1},{s})$-coloring; thus a circular $({2s+1},{s})$-coloring is also called a {\em $C_{2s+1}$-coloring} for convenience. For partial results of Conjecture \ref{CONJ:circular-conj}, Dvo\v{r}\'{a}k and Postle \cite{DvoPos} showed that every planar graph of  girth at least 10 is $C_5$-colorable. In \cite{CL20}, by duality from flow results, a simpler proof of the Dvo\v{r}\'{a}k and Postle's result was obtained, and it was extended to the next case that every planar graph of  girth at least 16 is $C_7$-colorable. Independently, Postle  and Smith-Roberge \cite{PosteSR} also proved that every planar  graph of girth at least 16 is $C_7$-colorable through the density of $C_7$-critical graphs. The current best general  result was due to Lov\'asz, Thomassen, Wu, Zhang \cite{LTWZ13}, from the dual of their more general flow results, that for each even $t$, $\chi_c(G)\le \frac{2t+2}{t}$ for every  planar graph $G$ of girth at least $3t$. For odd $t$, a recent flow result in \cite{LWZ19} also showed that $\chi_c(G)\le \frac{2t+2}{t}$ for every  planar graph $G$ of girth at least $3t+1$.

Another influential  coloring problem on planar graphs is the Steinberg's Conjecture (see \cite{S93}) from 1976, which asserts that every planar graph without cycles of length $4$ or $5$ is $C_3$-colorable.
We ask the following generalization on $C_{k}$-coloring.
\begin{question}\label{QEST:f(k)}
For any integer $k\geq 3$,  what is the smallest number $f(k)$ such that every planar graph of girth $k$ without cycles of length from $k+1$ to $f(k)$  is $C_k$-colorable?
\end{question}

As an approach to the Steinberg's Conjecture, Erd\H{o}s  (see \cite{S93}) asked to bound and  determine $f(3)$. Abbott and Zhou \cite{AZ91} first established that $f(3)\le 11$.  The bounds are progressively improved to $f(3)\le 9$ by Borodin \cite{Boro96} and by Sanders and Zhao \cite{SZ-GC95} independently, and to $f(3)\le 7$ by Borodin, Glebov,  Raspaud, and Salavatipour \cite{BGRS05}, i.e., every
planar graph without cycles of length from $4$ to $7$ is $3$-colorable. However, Steinberg's Conjecture has been disproved recently by  Cohen-Addad, Hebdige,  Kr\'al',  Li, and Salgado \cite{CHKLS17}, i.e., there exists a planar graph without cycles of length $4$ or $5$ that is not $3$-colorable. Those results combine to imply that $f(3)\in \{6,7\}$.

Our first main result of this paper describes  the existence of   $f(k)$ for all $k$.
\begin{theorem}\label{THMmain-p}
  The value $f(k)$ exists as a finite number if and only if $k$ is an odd prime.
  Moreover,  for any prime $p\ge 5$, $$p^2-\frac{5p-3}{2}\le f(p)\le 2p^2+2p-5.$$
\end{theorem}

We suspect that the lower bound in Theorem \ref{THMmain-p} is  close to the exact value of $f(p)$, and propose the following conjecture for  upper bound.
\begin{conjecture}\label{CONJ:f(p)}
  For any prime $p\ge 5$, $f(p)\le p(p-2)$. That is, every planar graph of girth $p$ without cycles of length from $p+1$ to  $p(p-2)$  is $C_p$-colorable.
\end{conjecture}
The following connection between Conjecture \ref{CONJ:circular-conj} and Conjecture \ref{CONJ:f(p)} is observed.
\begin{proposition}\label{PROP1.1-1.4}
 Let $p\ge 5$ be a prime. The truth of Conjecture \ref{CONJ:f(p)} implies the validity of Conjecture \ref{CONJ:circular-conj} for $t=p-1$. That is, Conjecture \ref{CONJ:f(p)} implies that every  planar graph of girth at least $2p-2$ is $C_p$-colorable.
\end{proposition}
Proposition \ref{PROP1.1-1.4} indicates that proving Conjecture \ref{CONJ:f(p)} may be difficult. But on the other hand, it  also suggests that Conjecture \ref{CONJ:f(p)} may provide a possible new approach to solve Conjecture \ref{CONJ:circular-conj} for $t=p-1$ with odd prime $p$. Particularly, the $p=5$ case of  Conjecture \ref{CONJ:f(p)} not only  implies that every  planar graph of girth $8$ is $C_5$-colorable, but also implies the Five Coloring Theorem as shown in Proposition \ref{PROPconto5color}.


\medskip
The fractional chromatic number of a graph is another well-known variation of the chromatic
number. For positive integers $a$ and $b$ with $a\ge b$, a {\em fractional  $(a:b)$-coloring} $\varphi$  of a graph $G$ is
a set coloring such that  each vertex assigns a $b$-element
subset of $\{1,\ldots, a\}$ satisfying $\varphi(u)\cap \varphi(v)=\emptyset$ whenever $uv\in E(G)$.
The {\em fractional chromatic number} of $G$, denoted by $\chi_f(G)$, is the infimum of the
fractions $\frac{a}{b}$ such that $G$ admits a fractional $(a:b)$-coloring. Notice that a fractional $(a :1)$-coloring of a graph $G$ coincides with an ordinary proper $a$-coloring of $G$.
The fractional coloring was first introduced by Hilton, Rado, and Scott \cite{planfr5} in 1973 to seek for a proof of the Four Color Problem.
Since then, it has been the focus of many intensive research efforts, see \cite{ScheinermanUllman2011}.
 For a graph $G$, let $\omega(G)$ and $\alpha(G)$ denote
the clique number and the independence number of $G$, respectively. It is well-known (cf. \cite{Zhu01, Zhu06}) that
$$\max\{\omega(G),\frac{|V(G)|}{\alpha(G)}\}\leq \chi_f(G)\leq \chi_c(G)\leq \lceil \chi_c(G)\rceil=\chi(G).$$

One may also consider the fractional coloring variations of Conjecture \ref{CONJ:circular-conj} and Question \ref{QEST:f(k)}. Analogous to Conjecture \ref{CONJ:circular-conj}, Naserasr \cite{Nase13} conjectured that every planar graph of girth at least $2s+2$ is fractional $(2s+1:s)$-colorable. It is proved for $K_4$-minor-free graphs in \cite{BFN17, GX16} that every $K_4$-minor-free graph of girth at least $2s$ is fractional $(2s+1:s)$-colorable.

Our second main result provides a fractional coloring result of Question \ref{QEST:f(k)}, which particularly confirms the fractional coloring version of Conjecture \ref{CONJ:f(p)} for prime $p\ge 11$ in a strong sense.

\begin{theorem}\label{THM:main-fraction-k}
For any odd integer $k\ge 5$, every planar graph of girth $k$ without cycles of length from $k+1$ to $\lfloor\frac{22k}{3}\rfloor$  is fractional  $(k:\frac{k-1}{2})$-colorable.
\end{theorem}
In a followup work \cite{HL20}, we also prove the remaining cases ($p=5,7$) of the fractional coloring version of Conjecture \ref{CONJ:f(p)} with some refined arguments and additional efforts.

The rest of this paper is organized as follows. We introduce some preliminaries and prove Proposition \ref{PROP1.1-1.4} in Section \ref{sect2}. The proof of Theorem \ref{THMmain-p} is  presented in Section \ref{sect3} and the proof of Theorem \ref{THM:main-fraction-k} is completed in Section \ref{sect4}. We end this paper with a few remarks  in Section \ref{sect5}.

\section{Preliminaries}\label{sect2}
We start with some basic notation and terminologies.  Let $G =(V(G),E(G))$ be a simple finite graph.
For a vertex $v\in V(G)$, the neighborhood $N_G(v)$ of a vertex $v$ is the set of vertices adjacent to $v$, and denote $d_G(v)=|N_G(v)|$.
The distance between two vertices $u$ and $v$, denoted by $d_G(u,v)$, is the length
of a shortest path from $u$ to $v$ in $G$. The subscript $G$ is often omitted if the graph $G$ is clear from the context.
For $S\subseteq V(G)$, $G - S$ denotes the graph obtained from $G$ by deleting all the vertices of $S$
together with all the edges incident to at least one vertex in $S$.
For a positive integer $i$, let $[i]=\{1,2,\ldots,i\}$.
We use $i^+$ to denote a number equal or greater than $i$. An $i$-vertex ($i^+$-vertex, resp.) is a vertex of degree $i$ (at least $i$, resp.). Similarly, in a plane graph, an $i$-face ($i^+$-face, resp.) is a face of degree $i$ (or at least $i$, resp.). In the rest of this paper, we usually   assume $k\ge 5$ is an odd integer and $p\ge 5$ is a prime implicitly.

A common method in graph coloring is to study certain coloring properties of typical graphs under given precoloring. This usually provides some reducible subgraphs and facilitates a discharging proof. We shall define precoloring properties for circular coloring and fractional coloring, respectively.  Let $H$ be a graph with a vertex subset $S\subset V(H)$. A {\em precoloring} $\omega$ assigns colors in $[k]$ to vertices in $S$ such that $H[S]$ is  properly $C_k$-colored. The graph $H$ is called $(\omega, S)$-colorable if the precoloring $\omega$ of $S$ can be extended to $V(H)$ to obtain a $C_k$-coloring of $H$.
Similarly, a {\em precoloring} $\varphi$ of $S$ assigns colors in ${[k]}\choose{\frac{k-1}{2}}$ to vertices in $S$ such that $H[S]$ is  properly fractional $(k:\frac{k-1}{2})$-colored. We say that  $H$ is {\em $\varphi_S$-colorable} if the precoloring $\varphi$ of $S$ can be extended to all vertices of $H$ to obtain a fractional $(k:\frac{k-1}{2})$-coloring.

%

%

We first observe the following fact on precoloring of $k$-cycle for $C_k$-coloring, which will be  useful.
\begin{lemma}\label{LEM:cycle-Ck-precoloring}
Let $G=v_0v_1\ldots v_{k-1}v_0$ be an odd cycle of length $k$. Let $\omega$ be a precoloring of $\{v_i,v_j\}\subseteq V(G)$. Then $G$ is $(\omega, \{v_i,v_j\})$-colorable if and only if
\begin{equation}\label{EQ:vi-vj}
  \omega(v_i)-\omega(v_j)\equiv \frac{k-1}{2}\cdot (i-j) ~~\text{or}~~\frac{k+1}{2}\cdot (i-j)\pmod{k}.
\end{equation}
\end{lemma}
\begin{proof} If $\omega$ can be extended to a $C_k$-coloring $\tilde{\omega}$ of $G$, then
  the $C_k$-coloring $\tilde{\omega}: V(G)\mapsto \{0,1,\ldots, k-1\}$ provides a coloring of $G$ such that
$$\text{either}~~~\tilde{\omega}(v_t)\equiv\frac{k-1}{2}\cdot t +\tilde{\omega}(v_0)\pmod{k} ~~\text{for each $0\le t\le k-1$}$$
$$\text{or}~~~\tilde{\omega}(v_t)\equiv\frac{k+1}{2}\cdot t+\tilde{\omega}(v_0)\pmod{k} ~~\text{for each $0\le t\le k-1$}.$$
Hence,  for $v_i,v_j\in V(G)$ we have Eq.~(\ref{EQ:vi-vj}).

Conversely, if Eq.~(\ref{EQ:vi-vj}) holds, then we can properly define a $C_k$-coloring of $G$ as above. This proves the lemma.
\end{proof}

In a graph $G$, a {\em $d$-$C_k$-replacement} operation on a given edge $e=xy\in E(G)$ is to replace the edge $e$ with a $k$-cycle $C_k=v_0v_1\ldots v_{k-1}v_0$ by identifying $x$ with $v_0$ and identifying $y$ with $v_{d}$. When $d$ is not explicitly stated, we just call it a $C_k$-replacement operation on the edge $e\in E(G)$.  Lemma \ref{LEM:cycle-Ck-precoloring} implies the following relation between  $C_k$-coloring and  $d$-$C_k$-replacement operation.

\begin{proposition}\label{PROP:d-k-replacement}
  Let $G$ be a graph, and let $G(d,k)$ be a graph obtained from $G$ by applying $d$-$C_k$-replacement operation on each edge of $G$. Assume that $d$ and $k$ are coprime, i.e., $gcd(d,k)=1$. Then $G$ is $C_k$-colorable if and only if $G(d,k)$ is $C_k$-colorable.
\end{proposition}
\begin{proof}
  Let $\varphi$ be a $C_k$-coloring of $G$.  Define a precoloring $\omega$ of $G(d,k)$ by coloring each vertex  $u\in V(G)\subset V(G(d,k))$ with $\omega(u)\equiv d\varphi(u)\pmod{k}$.  Since $\varphi(u)-\varphi(v)\in\{\frac{k-1}{2}, \frac{k+1}{2}\}$ for each edge $uv\in E(G)$,   we have, in the graph $G(d,k)$,  $$\omega(u)-\omega(v)\equiv d\varphi(u)-d\varphi(v)\equiv \frac{k-1}{2}\cdot d ~~\text{or}~~\frac{k+1}{2}\cdot d\pmod{k}.$$
  It follows from Lemma \ref{LEM:cycle-Ck-precoloring} that $\omega$ can be extended to a $C_k$-coloring of $G(d,k)$ by coloring each $k$-cycle of $G(d,k)$ properly.

  Conversely, assume that $G(d,k)$ admits a $C_k$-coloring $\omega$. Then for each edge $uv\in E(G)$, we have $\omega(u)-\omega(v)\in\{\frac{k-1}{2}\cdot d, \frac{k+1}{2}\cdot d\}\pmod{k}$ by Lemma \ref{LEM:cycle-Ck-precoloring}. Define $\varphi=d^{-1}\omega$ $\pmod{k}$. (Note that $d^{-1}$ exists in $\mathbb{Z}_k$ since $gcd(d,k)=1$.) Then $\varphi(u)-\varphi(v)\in\{\frac{k-1}{2}, \frac{k+1}{2}\}$ for each edge $uv\in E(G)$. That is, $\varphi$ restricted to $V(G)$ provides a proper $C_k$-coloring of $G$.
\end{proof}

Applying Lemma \ref{LEM:cycle-Ck-precoloring}, we  also show that $f(k)$ does not exist for nonprime $k$ by construction using $d$-$C_k$-replacement operations.

\begin{figure}

\minipage{0.45\textwidth}
\centering
\begin{tikzpicture}[scale=0.4]
\tikzstyle{mynodestyle} = [draw,shape=circle,outer sep=0,inner sep=1.2,minimum size=4,fill=black]

\tikzstyle{myedgestyle} = [line width=0.6pt, black]

\draw [myedgestyle] (3.6501,5.5583) circle (1.5);
\draw[myedgestyle]  (1,7) circle (1.5);
\draw [myedgestyle] (-1.7727,-6.4424) circle (1.5);
\draw[myedgestyle]  (1.2552,-6.4929) circle (1.5);
\draw [myedgestyle] (4.0579,-5.2974) circle (1.5);
\draw[myedgestyle]  (-4.3674,5.5545) circle (1.5);
\draw[myedgestyle]  (-4.4004,-4.8692) circle (1.5);

\draw[myedgestyle]  (5.3149,3.0092) circle (1.5);
\draw [myedgestyle] (6.116,0.098) circle (1.5);
\draw[myedgestyle]  (5.8871,-2.9208) circle (1.5);
\draw[myedgestyle]  (-6.1445,-2.4626) circle (1.5);
\draw [myedgestyle] (-7.0688,0.3941) circle (1.5);
\draw[myedgestyle]  (-6.3912,3.2962) circle (1.5);

\node [mynodestyle] at (2.2916,6.2919) {};
\node [mynodestyle] (v2) at (-0.5058,6.9054) {};
\node [mynodestyle] (v1) at (-2.9497,6.1151) {};
\node [mynodestyle] at (4.4755,4.2952) {};
\node [mynodestyle] at (2.2708,5.0335) {};
\node [mynodestyle] at (3.1548,4.1704) {};
\node [mynodestyle] at (3.1444,6.9782) {};
\node [mynodestyle] at (4.1843,6.9678) {};
\node [mynodestyle] at (4.7875,6.5415) {};
\node [mynodestyle] at (5.0786,5.9175) {};
\node [mynodestyle] at (5.1098,5.1687) {};
\node [mynodestyle] at (-0.0378,5.9383) {};
\node [mynodestyle] at (1.2572,5.5163) {};
\node [mynodestyle] at (1.4597,8.4237) {};
\node [mynodestyle] at (0.5341,8.4133) {};
\node [mynodestyle] at (-0.2042,7.9038) {};
\node [mynodestyle] at (2.0732,8.0286) {};
\node [mynodestyle] at (2.4268,7.3942) {};
\node [mynodestyle] at (5.7074,1.5684) {};
\node [mynodestyle] at (3.8203,3.2448) {};
\node [mynodestyle] at (4.2155,1.9969) {};
\node [mynodestyle] at (6.7284,3.4802) {};
\node [mynodestyle] at (5.4218,4.5344) {};
\node [mynodestyle] at (6.1498,4.2848) {};
\node [mynodestyle] at (6.788,2.7303) {};
\node [mynodestyle] at (6.4098,2.0281) {};
\node [mynodestyle] at (5.9783,-1.399) {};
\node [mynodestyle] at (4.8707,-0.7069) {};
\node [mynodestyle] at (4.7522,0.7225) {};
\node [mynodestyle] at (6.7114,1.5082) {};
\node [mynodestyle] at (7.2729,1.0402) {};
\node [mynodestyle] at (7.5953,0.2186) {};
\node [mynodestyle] at (7.4291,-0.5984) {};
\node [mynodestyle] at (6.8379,-1.186) {};
\node [mynodestyle] at (4.9807,-4.1193) {};
\node [mynodestyle] at (4.3715,-2.922) {};
\node [mynodestyle] at (4.9655,-1.7372) {};
\node [mynodestyle] at (7.3153,-3.2994) {};
\node [mynodestyle] at (7.2417,-2.298) {};
\node [mynodestyle] at (6.7634,-1.6845) {};
\node [mynodestyle] at (6.8257,-4.0555) {};
\node [mynodestyle] at (5.9319,-4.41) {};
\node [mynodestyle] at (2.6452,-5.8442) {};
\node [mynodestyle] at (2.666,-4.7211) {};
\node [mynodestyle] at (3.4693,-3.9128) {};
\node [mynodestyle] at (5.131,-6.3523) {};
\node [mynodestyle] at (5.4634,-4.8459) {};
\node [mynodestyle] at (5.5111,-5.7011) {};
\node [mynodestyle] at (4.2623,-6.7812) {};
\node [mynodestyle] at (3.2692,-6.5618) {};
\node [mynodestyle] at (-0.2666,-6.489) {};
\node [mynodestyle] at (0.2593,-5.3665) {};
\node [mynodestyle] at (1.5624,-5.0207) {};
\node [mynodestyle] at (1.6167,-7.9221) {};
\node [mynodestyle] at (2.3812,-7.5141) {};
\node [mynodestyle] at (2.7461,-6.7385) {};
\node [mynodestyle] at (0.6181,-7.8508) {};
\node [mynodestyle] at (-0.0391,-7.2813) {};
\node [mynodestyle] at (-3.0849,-5.6258) {};
\node [mynodestyle] at (-2.0556,-4.9662) {};
\node [mynodestyle] at (-0.8166,-5.2708) {};
\node [mynodestyle] at (-2.076,-7.9057) {};
\node [mynodestyle] at (-2.9304,-7.357) {};
\node [mynodestyle] at (-3.2397,-6.4179) {};
\node [mynodestyle] at (-1.1703,-7.7866) {};
\node [mynodestyle] at (-0.5266,-7.3105) {};
\node [mynodestyle] at (-5.2999,-3.6603) {};
\node [mynodestyle] at (-5.4338,-5.9334) {};
\node [mynodestyle] at (-5.8511,-4.4299) {};
\node [mynodestyle] at (-5.8303,-5.1787) {};
\node [mynodestyle] at (-4.7488,-6.3226) {};
\node [mynodestyle] at (-3.7991,-6.2533) {};
\node [mynodestyle] at (-3.8648,-3.4212) {};
\node [mynodestyle] at (-3.0485,-4.2341) {};
\node [mynodestyle] at (-6.5999,-0.9877) {};
\node [mynodestyle] at (-5.3727,-1.1853) {};
\node [mynodestyle] at (-4.6448,-2.1524) {};
\node [mynodestyle] at (-7.4734,-1.7781) {};
\node [mynodestyle] at (-7.5982,-2.7348) {};
\node [mynodestyle] at (-7.2652,-3.434) {};
\node [mynodestyle] at (-6.7039,-3.8163) {};
\node [mynodestyle] at (-6.0044,-3.9611) {};
\node [mynodestyle] at (-6.7247,1.8513) {};
\node [mynodestyle] at (-5.7575,1.165) {};
\node [mynodestyle] at (-5.7055,-0.3013) {};
\node [mynodestyle] at (-8.5842,0.4261) {};
\node [mynodestyle] at (-8.2543,1.3522) {};
\node [mynodestyle] at (-7.5991,1.8201) {};
\node [mynodestyle] at (-8.3766,-0.3717) {};
\node [mynodestyle] at (-7.6199,-0.9773) {};
\node [mynodestyle] at (-5.3528,4.4304) {};
\node [mynodestyle] at (-4.8848,3.276) {};
\node [mynodestyle] at (-5.4568,2.1009) {};
\node [mynodestyle] at (-7.6927,4.0248) {};
\node [mynodestyle] at (-7.8926,3.133) {};
\node [mynodestyle] at (-7.5263,2.3401) {};
\node [mynodestyle] at (-7.0375,4.6592) {};
\node [mynodestyle] at (-6.112,4.7528) {};
\node [mynodestyle] at (-3.9905,4.1184) {};
\node [mynodestyle] at (-2.8882,5.1999) {};
\node [mynodestyle] at (-5.28,6.6974) {};
\node [mynodestyle] at (-5.8104,5.9383) {};
\node [mynodestyle] at (-5.7896,5.1271) {};
\node [mynodestyle] at (-4.4169,7.051) {};
\node [mynodestyle] at (-3.4913,6.791) {};
\draw[myedgestyle]  (v1) edge (v2);

\node at (0.0635,7.103) {\small$z_0$};
\node at (2.947,6.2119) {\small$z_1$};
\node at (4.7994,3.7243) {\small$z_2$};
\node at (5.8265,1.0706) {\small$z_3$};
\node at (6.0449,-1.8853) {\small$z_4$};
\node at (4.5722,-4.4467) {\small$z_5$};
\node at (2.1555,-6.1146) {\small$z_6$};
\node at (-0.8251,-6.5826) {\small$z_7$};
\node at (-3.6865,-5.5218) {\small$z_8$};
\node at (-5.8284,-3.4108) {\small$z_9$};
\node at (-6.9335,-0.4757) {\small$z_{10}$};
\node at (-6.6047,2.3777) {\small$z_{11}$};
\node at (-4.9428,4.8008) {\small$z_{12}$};
\node at (-2.6214,6.6979) {\small$z_{13}$};
\node at (-0.8251,-10.5) {\small (a) Construction of $G$ for $k=9$ and $m=13$.};
\end{tikzpicture}
\endminipage\hfill
\minipage{0.55\textwidth}
\centering
\begin{tikzpicture}[scale=0.7]
\tikzstyle{mynodestyle} = [draw,shape=circle,outer sep=0,inner sep=1.2,minimum size=4.5,fill=black]
\tikzstyle{mynodestyle2} = [draw,shape=circle,outer sep=0,inner sep=1.2,minimum size=2,fill=black!50!]

\tikzstyle{myedgestyle} = [line width=0.6pt, black]
\node[mynodestyle] (v1) at (0,0) {};
\node[mynodestyle] (v2) at (0,4) {};
\node[mynodestyle] (v8) at (-1.9129,-3.489) {};
\node[mynodestyle] (v10) at (1.921,-3.489) {};
\node[mynodestyle] (v4) at (-3.1295,2.5277) {};
\node [mynodestyle](v14) at (3.1309,2.4944) {};
\node [mynodestyle](v6) at (-3.8941,-0.8961) {};
\node [mynodestyle](v12) at (3.8691,-0.9073) {};
\node [mynodestyle2](v3) at (-1.3027,2.9377) {};

\node [mynodestyle2](v15) at (1.3788,2.8712) {};
\node[mynodestyle2] (v13) at (3.063,0.7992) {};
\node[mynodestyle2] (v11) at (2.5644,-1.7715) {};
\node [mynodestyle2](v9) at (-0.0063,-3.1344) {};
\node[mynodestyle2] (v7) at (-2.5659,-1.7383) {};
\node[mynodestyle2] (v5) at (-3.0423,0.8657) {};

\node[mynodestyle2] (v17) at (-2.5839,4.2349) {};
\node[mynodestyle2] (v16) at (-1.1989,4.8443) {};
\node[mynodestyle2] (v29) at (1.2278,4.8) {};
\node [mynodestyle2](v28) at (2.7539,4.102) {};
\node [mynodestyle2](v27) at (4.4381,2.1628) {};
\node[mynodestyle2] (v26) at (4.973,0.2681) {};
\node[mynodestyle2] (v25) at (4.4492,-2.2139) {};
\node[mynodestyle2] (v24) at (3.3663,-3.6322) {};
\node[mynodestyle2] (v23) at (1.1059,-4.8067) {};
\node [mynodestyle2](v22) at (-0.9108,-4.8511) {};
\node[mynodestyle2] (v21) at (-3.3374,-3.6544) {};
\node[mynodestyle2] (v20) at (-4.4262,-2.2361) {};
\node[mynodestyle2] (v19) at (-4.9551,0.2681) {};
\node[mynodestyle2] (v18) at (-4.5356,2.041) {};

\draw[myedgestyle]   (v2) edge (v3);
\draw[myedgestyle]  (v3) edge (v4);
\draw [myedgestyle] (v4) edge (v5);
\draw[myedgestyle]  (v5) edge (v6);
\draw [myedgestyle] (v6) edge (v7);
\draw  [myedgestyle](v7) edge (v8);
\draw [myedgestyle] (v8) edge (v9);
\draw[myedgestyle]  (v9) edge (v10);
\draw [myedgestyle] (v10) edge (v11);
\draw [myedgestyle] (v11) edge (v12);
\draw[myedgestyle]  (v12) edge (v13);
\draw[myedgestyle]  (v14) edge (v13);
\draw[myedgestyle]  (v14) edge (v15);
\draw [myedgestyle] (v15) edge (v2);
\draw [myedgestyle] (v2) edge (v16);
\draw [myedgestyle] (v16) edge (v17);
\draw [myedgestyle] (v17) edge (v4);
\draw [myedgestyle] (v4) edge (v18);
\draw[myedgestyle]  (v18) edge (v19);
\draw[myedgestyle]  (v19) edge (v6);
\draw[myedgestyle]  (v6) edge (v20);
\draw [myedgestyle] (v20) edge (v21);
\draw [myedgestyle] (v21) edge (v8);
\draw [myedgestyle] (v8) edge (v22);
\draw [myedgestyle] (v22) edge (v23);
\draw [myedgestyle] (v23) edge (v10);
\draw[myedgestyle]  (v10) edge (v24);
\draw[myedgestyle]  (v24) edge (v25);
\draw [myedgestyle] (v25) edge (v12);
\draw [myedgestyle] (v12) edge (v26);
\draw[myedgestyle]  (v26) edge (v27);
\draw[myedgestyle]  (v27) edge (v14);
\draw [myedgestyle] (v14) edge (v28);
\draw[myedgestyle]  (v28) edge (v29);
\draw[myedgestyle]  (v2) edge (v29);


\node [mynodestyle2](v35) at (-0.0197,1.3382) {};
\node[mynodestyle2] (v42) at (1.0219,0.8396) {};
\node[mynodestyle2] (v53) at (1.2989,-0.2241) {};
\node[mynodestyle2] (v64) at (0.7227,-1.1106) {};
\node[mynodestyle2] (v75) at (-0.7288,-1.1106) {};
\node[mynodestyle2] (v86) at (-1.2828,-0.202) {};
\node [mynodestyle2] (v97) at (-0.9504,0.9504) {};


\node[mynodestyle2] (v31) at (-0.0086,2.6014) {};
\node[mynodestyle2] (v46) at (1.9859,1.6928) {};
\node [mynodestyle2] (v57) at (2.5289,-0.5011) {};
\node[mynodestyle2] (v68) at (1.3322,-2.1964) {};
\node [mynodestyle2] (v79) at (-1.2828,-2.2186) {};
\node[mynodestyle2] (v90) at (-2.5238,-0.5011) {};
\node[mynodestyle2] (v101) at (-1.9477,1.715) {};


\node[mynodestyle2] (v30) at (-0.1305,3.3106) {};
\node[mynodestyle2] (v33) at (0.3238,3.4989) {};
\node[mynodestyle2] (v32) at (0.3571,3.0557) {};
\node[mynodestyle2] (v34) at (-0.1415,1.9588) {};
\node[mynodestyle2] (v37) at (0.3903,2.2579) {};
\node[mynodestyle2] (v36) at (0.3792,1.7704) {};
\node[mynodestyle2] (v38) at (-0.1526,0.7066) {};
\node [mynodestyle2](v40) at (0.2906,0.9726) {};
\node[mynodestyle2] (v39) at (0.2906,0.5515) {};
\draw [myedgestyle] (v2) edge (v30);
\draw [myedgestyle] (v30) edge (v31);
\draw [myedgestyle] (v31) edge (v32);
\draw[myedgestyle]  (v32) edge (v33);
\draw[myedgestyle]  (v33) edge (v2);
\draw [myedgestyle] (v31) edge (v34);
\draw [myedgestyle] (v34) edge (v35);
\draw[myedgestyle]  (v35) edge (v36);
\draw[myedgestyle]  (v36) edge (v37);
\draw[myedgestyle]  (v37) edge (v31);
\draw[myedgestyle]  (v35) edge (v38);
\draw [myedgestyle] (v38) edge (v1);
\draw[myedgestyle]  (v1) edge (v39);
\draw[myedgestyle]  (v39) edge (v40);
\draw[myedgestyle]  (v40) edge (v35);
\node[mynodestyle2] (v49) at (2.4956,2.1471) {};
\node[mynodestyle2] (v51) at (2.4956,1.7039) {};
\node [mynodestyle2](v50) at (2.9277,2.0363) {};
\node [mynodestyle2](v45) at (1.4429,1.3714) {};
\node [mynodestyle2](v48) at (1.5205,0.7842) {};
\node[mynodestyle2] (v47) at (1.8419,1.0834) {};
\node[mynodestyle2] (v44) at (0.5454,0.1415) {};
\node[mynodestyle2] (v43) at (0.8779,0.3853) {};
\node[mynodestyle2] (v41) at (0.5455,0.474) {};
\draw [myedgestyle] (v1) edge (v41);
\draw[myedgestyle]  (v41) edge (v42);
\draw[myedgestyle]  (v42) edge (v43);
\draw [myedgestyle] (v43) edge (v44);
\draw [myedgestyle] (v44) edge (v1);
\draw[myedgestyle]  (v42) edge (v45);
\draw [myedgestyle] (v45) edge (v46);
\draw [myedgestyle] (v46) edge (v47);
\draw [myedgestyle] (v47) edge (v48);
\draw[myedgestyle]  (v48) edge (v42);
\draw [myedgestyle] (v46) edge (v49);
\draw [myedgestyle] (v49) edge (v14);
\draw[myedgestyle]  (v14) edge (v50);
\draw [myedgestyle] (v50) edge (v51);
\draw[myedgestyle]  (v51) edge (v46);

\node[mynodestyle2] (v52) at (0.7227,-0.0025) {};
\node[mynodestyle2] (v56) at (2.0192,-0.1355) {};
\node[mynodestyle2] (v60) at (3.2602,-0.4679) {};
\node[mynodestyle2] (v55) at (0.4346,-0.2352) {};
\node[mynodestyle2] (v54) at (0.8114,-0.346) {};
\node[mynodestyle2] (v59) at (1.6757,-0.4679) {};
\node [mynodestyle2](v58) at (2.0967,-0.5565) {};
\node[mynodestyle2] (v62) at (2.8834,-0.7117) {};
\node[mynodestyle2] (v61) at (3.4153,-0.8668) {};
\draw[myedgestyle]  (v1) edge (v52);
\draw [myedgestyle] (v52) edge (v53);
\draw[myedgestyle]  (v53) edge (v54);
\draw [myedgestyle] (v54) edge (v55);
\draw [myedgestyle] (v55) edge (v1);
\draw[myedgestyle]  (v53) edge (v56);
\draw [myedgestyle] (v56) edge (v57);
\draw[myedgestyle]  (v57) edge (v58);
\draw [myedgestyle] (v58) edge (v59);
\draw[myedgestyle]  (v59) edge (v53);
\draw[myedgestyle]  (v57) edge (v60);
\draw [myedgestyle] (v60) edge (v12);
\draw [myedgestyle] (v12) edge (v61);
\draw [myedgestyle] (v61) edge (v62);
\draw[myedgestyle]  (v62) edge (v57);

\node [mynodestyle2](v63) at (0.4679,-0.5565) {};
\node [mynodestyle2](v67) at (1.166,-1.5427) {};
\node [mynodestyle2](v71) at (1.7532,-2.684) {};
\node[mynodestyle2] (v66) at (0.0801,-0.5676) {};
\node[mynodestyle2] (v65) at (0.3349,-0.8889) {};
\node[mynodestyle2] (v70) at (0.6341,-1.4873) {};
\node [mynodestyle2](v69) at (0.8668,-1.8529) {};
\node[mynodestyle2] (v73) at (1.2657,-2.5732) {};
\node[mynodestyle2] (v72) at (1.5649,-3.0386) {};
\draw[myedgestyle]  (v1) edge (v63);
\draw [myedgestyle] (v63) edge (v64);
\draw [myedgestyle] (v64) edge (v65);
\draw[myedgestyle]  (v65) edge (v66);
\draw [myedgestyle] (v66) edge (v1);
\draw [myedgestyle] (v64) edge (v67);
\draw [myedgestyle] (v67) edge (v68);
\draw [myedgestyle] (v68) edge (v69);
\draw [myedgestyle] (v69) edge (v70);
\draw[myedgestyle]  (v70) edge (v64);
\draw[myedgestyle]  (v68) edge (v71);
\draw [myedgestyle] (v71) edge (v10);
\draw [myedgestyle] (v10) edge (v72);
\draw [myedgestyle] (v72) edge (v73);
\draw[myedgestyle]  (v73) edge (v68);
\node [mynodestyle2](v74) at (-0.208,-0.7006) {};
\node[mynodestyle2] (v77) at (-0.4629,-0.3238) {};
\node[mynodestyle2] (v76) at (-0.7066,-0.6341) {};
\node[mynodestyle2] (v78) at (-0.7731,-1.8419) {};
\node[mynodestyle2] (v81) at (-1.1942,-1.3765) {};
\node[mynodestyle2] (v80) at (-1.3715,-1.7311) {};
\node[mynodestyle2] (v82) at (-1.3936,-2.8723) {};
\node[mynodestyle2] (v84) at (-1.6928,-2.4624) {};
\node [mynodestyle2](v83) at (-1.859,-2.8502) {};
\draw[myedgestyle]  (v1) edge (v74);
\draw[myedgestyle]  (v74) edge (v75);
\draw[myedgestyle]  (v75) edge (v76);
\draw[myedgestyle]  (v76) edge (v77);
\draw [myedgestyle] (v77) edge (v1);
\draw[myedgestyle]  (v75) edge (v78);
\draw [myedgestyle] (v78) edge (v79);
\draw [myedgestyle] (v79) edge (v80);
\draw [myedgestyle] (v80) edge (v81);
\draw [myedgestyle] (v81) edge (v75);
\draw [myedgestyle] (v79) edge (v82);
\draw [myedgestyle] (v82) edge (v8);
\draw [myedgestyle] (v8) edge (v83);
\draw [myedgestyle] (v83) edge (v84);
\draw[myedgestyle]  (v84) edge (v79);
\node[mynodestyle2] (v85) at (-0.6512,-0.1909) {};
\node[mynodestyle2] (v88) at (-0.4629,0.1969) {};
\node[mynodestyle2] (v87) at (-0.9615,0.1748) {};
\node [mynodestyle2](v89) at (-1.8147,-0.5122) {};
\node [mynodestyle2](v92) at (-1.726,0.0307) {};
\node[mynodestyle2] (v91) at (-2.2247,-0.0579) {};
\node [mynodestyle2](v93) at (-3.0557,-0.8114) {};
\node [mynodestyle2](v95) at (-2.9892,-0.3238) {};
\node [mynodestyle2](v94) at (-3.3992,-0.4346) {};
\draw[myedgestyle]  (v1) edge (v85);
\draw[myedgestyle]  (v85) edge (v86);
\draw[myedgestyle]  (v86) edge (v87);
\draw [myedgestyle] (v87) edge (v88);
\draw [myedgestyle] (v88) edge (v1);
\draw[myedgestyle]  (v86) edge (v89);
\draw [myedgestyle] (v89) edge (v90);
\draw [myedgestyle] (v90) edge (v91);
\draw[myedgestyle]  (v91) edge (v92);
\draw [myedgestyle] (v92) edge (v86);
\draw[myedgestyle]  (v90) edge (v93);
\draw[myedgestyle]  (v93) edge (v6);
\draw  [myedgestyle](v6) edge (v94);
\draw [myedgestyle] (v94) edge (v95);
\draw [myedgestyle] (v95) edge (v90);
\node[mynodestyle2] (v96) at (-0.5515,0.4186) {};
\node[mynodestyle2] (v99) at (-0.3188,0.5404) {};
\node[mynodestyle2] (v98) at (-0.5294,0.7953) {};
\draw [myedgestyle] (v1) edge (v96);
\draw [myedgestyle] (v96) edge (v97);
\draw[myedgestyle]  (v97) edge (v98);
\draw [myedgestyle] (v1) edge (v99);
\draw [myedgestyle] (v99) edge (v98);
\node[mynodestyle2] (v100) at (-1.5931,1.2053) {};
\node[mynodestyle2] (v103) at (-1.0945,1.3493) {};
\node[mynodestyle2] (v102) at (-1.4379,1.6817) {};
\node[mynodestyle2] (v104) at (-2.6236,1.9366) {};
\node[mynodestyle2] (v106) at (-2.1582,2.1249) {};
\node[mynodestyle2] (v105) at (-2.546,2.4352) {};
\draw [myedgestyle] (v97) edge (v100);
\draw[myedgestyle]  (v100) edge (v101);
\draw[myedgestyle]  (v101) edge (v102);
\draw [myedgestyle] (v102) edge (v103);
\draw [myedgestyle] (v103) edge (v97);
\draw [myedgestyle] (v101) edge (v104);
\draw [myedgestyle] (v104) edge (v4);
\draw[myedgestyle]  (v4) edge (v105);
\draw [myedgestyle] (v105) edge (v106);
\draw [myedgestyle] (v106) edge (v101);
\node at (-0.2,-6) {\small (b) Construction of $H_p$ for $p=5$.};
\end{tikzpicture}
\endminipage\hfill
\caption{Constructions in Propositions \ref{PROP:nonprimek} and \ref{PROP:p-prime-quadratic-bound}.}\label{Fig:constra}
\end{figure}

\begin{proposition}\label{PROP:nonprimek}
  Let $k>0$ be  an odd nonprime integer. Then $f(k)$ does not exist. That is, for any integer $m>k$ there exist planar graphs of girth $k$ without cycles of length from $k+1$ to $m$ admitting no $C_k$-coloring.
\end{proposition}
\begin{proof}
   Denote $k=st$, where $s,t$ are positive integers with $t\ge s>1$.   Take an $(m+1)$-cycle $z_0z_1z_2\ldots z_{m}z_0$. For each $0\le i\le m-1$, apply $s$-$C_k$-replacement operation on the edge $z_iz_{i+1}$. Let $G$ be the resulting graph. Then  $G$ is a planar graph of girth $k$  without cycles of length from $k+1$ to $ms$. See Figure \ref{Fig:constra}(a) for the construction of $G$ when $k=9$ and $m=13$.

   It is routine  to check that $G$ is not $C_k$-colorable.  To see this,  suppose for a contradiction that $\omega: V(G) \mapsto \{0,1,\ldots,{k-1}\}$ is a $C_k$-coloring of $G$. By Lemma \ref{LEM:cycle-Ck-precoloring}, for each $0\le i\le m-1$, we have
   $$\omega(z_{i+1})-\omega(z_i)\equiv \frac{k-1}{2}\cdot s ~~\text{or}~~\frac{k+1}{2}\cdot s\pmod{k}.$$
    Thus $\omega(z_{i+1})-\omega(z_i)$ is a multiple of $s$ since $k=st$. This implies that $$\omega(z_{m})-\omega(z_0)=\sum_{i=0}^{m-1}(\omega(z_{i+1})-\omega(z_i))~\text{is a multiple of $s$.}$$
    On the other hand, as $z_mz_0$ is an edge in $E(G)$, we must have $|\omega(z_{m})-\omega(z_0)|\in\{\frac{k-1}{2},\frac{k+1}{2}\}$. But as $k=st$, neither $\frac{k-1}{2}$ nor $\frac{k+1}{2}$ is a multiple of $s$, a contradiction. This completes the proof.
\end{proof}

In contrast, we will show below in Theorem \ref{THM-main1-p} that $f(p)$ exists as a quadratic function of $p$ for odd prime $p$.   Now we give a low bound of $f(p)$ with similar arguments as Proposition \ref{PROP:nonprimek}.
\begin{proposition}\label{PROP:p-prime-quadratic-bound}
   For any prime  $p\ge 5$, there exist planar graphs of girth $p$ without cycles of length from $p+1$ to $p^2-\frac{5p-1}{2}$ admitting no $C_p$-coloring. That is, $f(p)\ge p^2-\frac{5}{2}p+\frac{3}{2}.$
\end{proposition}
\begin{proof}
  Construct a graph $W_p$ from a $(2p-3)$-cycle $z_0z_1\ldots z_{2p-4}z_0$ by adding a new center vertex $x$ connecting each $z_i$ with a new path of length $p-2$ for $0\le i\le 2p-4$. This graph $W_p$ was constructed by DeVos (see \cite{BKKW2004}) to show the tightness of  Conjecture \ref{CONJ:circular-conj}, i.e., $W_p$ is a planar graph of girth $2p-3$ without $C_p$-coloring. To see that $W_p$ is not $C_p$-colorable, we prove by contradiction. Suppose to the contrary that $\omega$ is a $C_p$-coloring of $W_p$.  If $\omega(x)=\omega(z_i)$ for some $i$, then after identifying $x$ and $z_i$ in the path of length $p-2$ between $x$ and $z_i$, we obtain a  $C_p$-coloring of $(p-2)$-cycle, a contradiction. So $\omega(x)\neq \omega(z_i)$ for each $0\le i\le 2p-4$. Hence the $(2p-3)$-cycle $z_0z_1\ldots z_{2p-4}z_0$ admits a $C_p$-coloring with colors $\{0,1,\ldots, p-1\}\setminus\{\omega(x)\}$. This provides a homomorphism from the $(2p-3)$-cycle to a path of length $p-2$; in particular, it indicates that the  $(2p-3)$-cycle is $2$-colorable, a contradiction.

  Construct a graph $H_p$ from $W_p$ by  applying $(\frac{p-1}{2})$-$C_p$-replacement operation on each edge of $W_p$. See Figure \ref{Fig:constra}(b) for the construction of $H_5$. Since $W_p$ is not $C_p$-colorable, we obtain that $H_p$ is not $C_p$-colorable by Proposition \ref{PROP:d-k-replacement}. As $W_p$ has girth $2p-3$, $H_p$ is of girth $p$ and without cycles of length from $p+1$ to $(2p-3)\frac{p-1}{2}-1$.
\end{proof}

Next, we shall prove Proposition \ref{PROP1.1-1.4} using analogous approaches.

\begin{proposition}(restatement of Proposition \ref{PROP1.1-1.4})
Let $p\ge 5$ be a prime. If $f(p)\le p(p-2)$, then every  planar graph of girth at least $2p-2$ is $C_p$-colorable.
\end{proposition}
\begin{proof}
  Assume that $f(p)\le p(p-2)$. That is, every planar graph of girth $p$ without cycles of length from $p+1$ to $p(p-2)$  is $C_p$-colorable. Let $G$ be a planar graph of girth at least $2p-2$. Apply the $(\frac{p-1}{2})$-$C_p$-replacement operation on each edge of $G$ to obtain a graph $G(\frac{p-1}{2},p)$. Then $G(\frac{p-1}{2},p)$ is a planar graph of girth $p$ without cycles of length from $p+1$ to $p(p-2)$. Since $f(p)\le p(p-2)$, we know that $G(\frac{p-1}{2},p)$ is $C_p$-colorable. Hence $G$ is $C_p$-colorable as well by Proposition \ref{PROP:d-k-replacement}.
\end{proof}

Similar arguments also show that the $p=5$ case of Conjecture \ref{CONJ:f(p)} is stronger than the Five Color Theorem.

\begin{proposition}\label{PROPconto5color}
The truth of $f(5)\le 17$ implies that every  planar graph is 5-colorable.
\end{proposition}
\begin{proof}

Assume that $f(5)\le 17$, i.e., every planar graph of girth $5$ without cycles of length from $6$ to $17$ is $C_5$-colorable.
Let $G$ be a planar graph, and $H$ be the graph obtained from $G$ by replacing each edge with a path of length $3$.
Let $F$ be the graph obtained from $H$ by  applying $2$-$C_5$-replacement operation on each edge of $H$. Then by construction $F$ is a planar graph of girth $5$ without cycles of length from $6$ to $17$, and  hence $F$ is $C_5$-colorable by $f(5)\le 17$.
Now, by Proposition \ref{PROP:d-k-replacement} and Lemma \ref{LEM:cycle-Ck-precoloring},  the $C_5$-coloring $\omega$ of $F$ induces a proper $5$-coloring of $G$, since $\omega(u)\neq \omega(v)$ whenever $uv\in E(G)$.
\end{proof}


At the end of this section, we  define some  graphs, serving for reducible configurations in  later proofs.
\begin{definition}\label{DEF:necklace}
 Let  $G$ be a graph.

 (i)
A {\bf thread} in $G$ is a path  whose internal vertices are $2$-vertices in $G$.
The end vertices of the path are called the end vertices of the thread. A thread with end vertices $x,y$ is also called an $(x,y)$-thread, denoted by $T(x,y)$.
An $s$-thread is a thread with $s$ internal vertices.
A {\bf $(k_1,k_2,\ldots,k_t)$-thread} $T_x$ in $G$ is a subgraph consisting of distinct $k_1$-thread, $k_2$-thread, $\ldots$, $k_t$-thread which share a common end vertex $x$,
where $t\geq 3$.  The common end vertex $x$ is called a {\bf $(k_1,k_2,\ldots,k_t)$-vertex}.
Let $y_i$ be the other end vertex of the  $k_i$-thread, and  define $\{y_1,y_2,\ldots, y_t\}$ to be the end vertices of $T_x$. If $z$ is a $2$-vertex of an $(x,y)$-thread, then we say $x$ and $z$ are {\bf weakly adjacent}.

(ii) An {\bf $s$-necklace} in $G$ is a subgraph obtained from an $s$-thread by applying $C_k$-replacement operations on some edges. A vertex $z$ is an end vertex of the $s$-necklace if and only if $z$ is an end vertex of the $s$-thread. A necklace with end vertices $x,y$ is also called an $(x,y)$-necklace, denoted by $N(x,y)$.
A {\bf $(k_1,k_2,\ldots,k_t)$-necklace} $N_x$ is a subgraph obtained from a $(k_1,k_2,\ldots,k_t)$-thread $T_x$ by applying $C_k$-replacement operations on some edges. The vertex $x$ is called the center vertex of $N_x$. A vertex $z$ is an end vertex of the $(k_1,k_2,\ldots,k_t)$-necklace if and only if $z$ is an end vertex of the $(k_1,k_2,\ldots,k_t)$-thread.
A {\bf $(k_1,k_2;k_3)$-bull-necklace} is a subgraph obtained from a $(k_1,k_2,k_3)$-thread by applying $C_k$-replacement operations on some edges of the $k_3$-thread.
A {\bf $(k_1,k_2,\ldots,k_t)$-crown-necklace} is obtained from a $(k_1,k_2,\ldots,k_t)$-necklace by replacing the center vertex with a $k$-cycle.
A vertex $z$ is an end vertex of the $(k_1,k_2,\ldots,k_t)$-crown-necklace if and only if $z$ is an end vertex of the $(k_1,k_2,\ldots,k_t)$-necklace.
See Figure \ref{FIG: thread-necklace}  for some examples.
\end{definition}

\begin{figure}[tb]
\begin{center}
\begin{tikzpicture}[scale=0.6]
\tikzstyle{mynodestyle} = [draw,shape=circle,outer sep=0,inner sep=1.2,minimum size=3.5,fill=black]
\tikzstyle{myedgestyle} = [line width=0.6pt, black]

\node  [draw,shape=circle,outer sep=0,inner sep=1.2,minimum size=6,fill=black] (v2) at (-4.5,0.5) {};
\node[mynodestyle] (v6) at (-6.5,2.5) {};
\node [mynodestyle](v5) at (-4,4) {};
\node[mynodestyle] (v7) at (-6.5,-1.5) {};
\node[mynodestyle] (v8) at (-5,-2) {};
\node[mynodestyle] (v9) at (-3,-1) {};
\node [mynodestyle](v3) at (-2.5,1) {};
\node  [draw,shape=circle,outer sep=0,inner sep=1.2,minimum size=6,fill=black] (v30) at (4.5,0.5) {};
\node[mynodestyle] (v36) at (6.2969,2.6871) {};

\node [mynodestyle] (v40) at (6,-2.5) {};
\node(v1) at (13.5,0.5) {};

\node [mynodestyle] (v69) at (11.793,3.8867) {};
\node [mynodestyle] (v65) at (11,0) {};
\node[mynodestyle] (v70) at (12.0759,-1.9757) {};

\node[mynodestyle] (v57) at (15.8667,0.8062) {};
\node [mynodestyle] at (15.4694,2.5396) {};
\draw [line width=0.9pt, black]  (v1) circle ();
\draw [myedgestyle]  (16.485,2.8133) circle (1);
\draw  [myedgestyle] (12.8612,-1.3913) circle (1);

\draw  [myedgestyle] (11.6453,2.9101) circle (1);

\draw  [myedgestyle] (14.5551,2.2898) circle (1);
\draw [myedgestyle]  (2.7009,2.9465) circle (1);
\draw [myedgestyle]  (2.3282,-1.3308) circle (1);
\draw  [myedgestyle] (16.8109,0.4262) circle (1);
\draw  [myedgestyle] (5.0096,-0.3712) circle (1);
\draw [myedgestyle]  (3.88,1.3011) circle (1);

\draw [myedgestyle] (v2) edge (v3);
\draw [myedgestyle] (v2) edge (v5);
\draw [myedgestyle] (v2) edge (v6);
\draw [myedgestyle] (v2) edge (v7);
\draw [myedgestyle] (v2) edge (v8);
\draw[myedgestyle]  (v2) edge (v9);

\node [mynodestyle] at (-5.8008,1.8196) {};
\node [mynodestyle] at (-5.1828,1.2016) {};
\node [mynodestyle] (v10) at (-6.8429,3.1647) {};
\node [mynodestyle] at (-4.0959,3.0556) {};
\node [mynodestyle] at (-4.2898,1.9408) {};

\node [mynodestyle] at (-3.5385,0.7412) {};
\node [mynodestyle] (v12) at (-1.6603,1.1168) {};

\node [mynodestyle] at (-5.1016,-0.1434) {};
\node [mynodestyle] at (-5.8044,-0.7614) {};
\node [mynodestyle] at (-4.8108,-1.2945) {};
\node [mynodestyle] at (-4.6412,-0.3978) {};
\node [mynodestyle] (v11) at (-5.1259,-2.8213) {};

\draw [myedgestyle] (v6) edge (v10);
\draw [myedgestyle] (v8) edge (v11);
\draw [myedgestyle] (v3) edge (v12);

\node (v22) at (-1.1,0.7014) {};
\node (v21) at (-1.1888,1.8525) {};
\node (v16) at (-4.2655,4.5703) {};
\node (v17) at (-3.4537,4.6066) {};
\node (v15) at (-6.6133,3.961) {};
\node (v14) at (-7.3676,3.4676) {};
\node (v13) at (-7.409,3.0677) {};
\node (v29) at (-7.0153,-0.8375) {};
\node (v28) at (-7.1405,-2.1623) {};
\node (v27) at (-5.6812,-2.894) {};
\node (v25) at (-4.5907,-3.1121) {};
\node (v24) at (-2.87,-1.658) {};
\node (v23) at (-2.4147,-1.4029) {};
\node (v26) at (-5.1359,-3.4272) {};
\draw [myedgestyle] (v13) edge (v10);
\draw [myedgestyle] (v14) edge (v10);
\draw [myedgestyle] (v15) edge (v10);
\draw [myedgestyle] (v16) edge (v5);
\draw [myedgestyle] (v17) edge (v5);

\draw [myedgestyle] (v21) edge (v12);
\draw [myedgestyle] (v22) edge (v12);
\draw [myedgestyle] (v23) edge (v9);
\draw [myedgestyle] (v24) edge (v9);
\draw [myedgestyle] (v25) edge (v11);
\draw [myedgestyle] (v26) edge (v11);
\draw [myedgestyle] (v27) edge (v11);
\draw [myedgestyle] (v28) edge (v7);
\draw [myedgestyle] (v29) edge (v7);
\node [mynodestyle] (v32) at (3.1402,-0.6644) {};
\node [mynodestyle] (v31) at (3.6855,-0.1918) {};
\node [mynodestyle] (v39) at (5.6,-1.1612) {};
\node [mynodestyle] (v41) at (6.2543,-3.3545) {};
\node [mynodestyle] at (3.2735,2.1468) {};
\node [mynodestyle] (v35) at (5.7575,1.7348) {};
\node [mynodestyle] (v37) at (6.5694,3.3585) {};
\draw [myedgestyle] (v30) edge (v31);
\draw [myedgestyle] (v31) edge (v32);
\node [mynodestyle] (v33) at (2.7161,-2.2154) {};
\node [mynodestyle] (v34) at (2.6191,-3.0394) {};
\draw [myedgestyle] (v33) edge (v34);
\draw [myedgestyle] (v30) edge (v35);
\draw [myedgestyle] (v35) edge (v36);
\draw [myedgestyle] (v36) edge (v37);
\node [mynodestyle] (v38) at (6.4967,1.2986) {};
\draw [myedgestyle] (v30) edge (v38);
\node [mynodestyle] (v45) at (3.2856,3.8069) {};
\draw [myedgestyle] (v39) edge (v40);
\draw [myedgestyle] (v40) edge (v41);
\node (v44) at (3.0069,4.4733) {};
\node (v46) at (3.3946,4.4976) {};
\node (v47) at (3.843,4.4006) {};
\node (v42) at (2.1697,-3.6006) {};
\node (v43) at (2.9827,-3.6574) {};
\node (v53) at (5.6999,-3.9238) {};
\node (v54) at (6.3178,-3.9884) {};
\node (v55) at (7.0904,-3.8634) {};
\node (v52) at (6.8602,0.717) {};
\node (v51) at (7.1268,1.3471) {};
\node (v48) at (6.0997,3.9283) {};
\node (v49) at (6.7945,3.9444) {};
\node (v50) at (7.2358,3.4797) {};
\draw [myedgestyle] (v34) edge (v42);
\draw [myedgestyle] (v34) edge (v43);
\draw [myedgestyle] (v44) edge (v45);
\draw [myedgestyle] (v46) edge (v45);
\draw [myedgestyle] (v47) edge (v45);
\draw [myedgestyle] (v48) edge (v37);
\draw [myedgestyle] (v49) edge (v37);
\draw [myedgestyle] (v50) edge (v37);
\draw [myedgestyle] (v51) edge (v38);
\draw [myedgestyle] (v52) edge (v38);
\draw [myedgestyle] (v41) edge (v53);
\draw [myedgestyle] (v41) edge (v54);
\draw [myedgestyle] (v41) edge (v55);
\node [mynodestyle] (v64) at (12.4826,0.3534) {};
\node [mynodestyle] at (13.1491,-0.4221) {};
\node [mynodestyle] (v62) at (14.1185,-0.2524) {};
\node [mynodestyle] at (14.4699,0.6321) {};
\node [mynodestyle] (v67) at (12.6765,1.1532) {};
\node [mynodestyle] at (13.3551,1.5167) {};
\node [mynodestyle] (v56) at (14.0336,1.3834) {};
\node [mynodestyle] (v60) at (16.8691,3.722) {};
\node [mynodestyle] (v58) at (17.4144,-0.3009) {};
\node [mynodestyle] (v59) at (17.8748,-1.1855) {};
\node [mynodestyle] (v71) at (11.1376,-2.482) {};
\draw  (v56) edge (v57);
\draw  (v58) edge (v59);
\node [mynodestyle] (v61) at (17.5719,4.1098) {};
\draw  (v60) edge (v61);
\node [mynodestyle] (v68) at (12.2645,2.171) {};
\node [mynodestyle] (v72) at (10.6166,-2.9667) {};
\node [mynodestyle] (v66) at (10.2773,-0.1313) {};
\node [mynodestyle] (v63) at (15.0152,-1.3066) {};
\draw  (v62) edge (v63);
\draw  (v64) edge (v65);
\draw  (v65) edge (v66);
\draw  (v67) edge (v68);
\draw  (v70) edge (v71);
\draw  (v71) edge (v72);
\node (v73) at (11.6102,4.5824) {};
\node (v74) at (12.0222,4.5097) {};
\node (v75) at (12.5553,4.3158) {};
\node (v76) at (18.314,4.7522) {};
\node (v77) at (18.1201,3.6783) {};
\node (v82) at (15.0151,-2.2276) {};
\node (v81) at (15.4392,-1.9367) {};
\node (v80) at (15.9724,-1.6338) {};
\node (v78) at (17.6757,-1.9082) {};
\node (v79) at (18.5725,-1.8477) {};
\node (v85) at (10.0955,-3.6453) {};
\node (v84) at (10.7377,-3.6816) {};
\node (v83) at (11.3921,-3.718) {};
\node (v86) at (10.0228,-0.8825) {};
\node (v87) at (9.6956,-0.519) {};
\draw[myedgestyle]   (v73) edge (v69);
\draw[myedgestyle]   (v74) edge (v69);
\draw[myedgestyle]   (v75) edge (v69);
\draw [myedgestyle]  (v76) edge (v61);
\draw [myedgestyle]  (v77) edge (v61);
\draw [myedgestyle]  (v78) edge (v59);
\draw [myedgestyle]  (v79) edge (v59);
\draw[myedgestyle]   (v80) edge (v63);
\draw[myedgestyle]   (v81) edge (v63);
\draw [myedgestyle]  (v82) edge (v63);
\draw[myedgestyle]   (v83) edge (v72);
\draw [myedgestyle]  (v84) edge (v72);
\draw [myedgestyle]  (v85) edge (v72);
\draw[myedgestyle]   (v86) edge (v66);
\draw[myedgestyle]   (v87) edge (v66);

\node at (-3.1925,6.3) {\small (a) a $5$-necklace;};

\node at (10.955,6.3) {\small (b) a $(2,2;4)$-bull-necklace;};

\node at (-4.6152,-5.3) {\small (c) a $(0,2,2,3,2,3)$-thread;};

\node at (4.461,-5.3) {\small (d) a $(2,0,2,1,3)$-necklace;};

\node at (14.3,-5.3) {\small (e) a $(2,0,2,2,1,1)$-crown-necklace.};

\node at (-2.6685,-0.7505)  {\small $z_1$};
\node at (-1.8662,1.546)  {\small $z_2$};
\node at (-4.4736,4.0431) {\small $z_3$};
\node at (-7.1111,2.8497) {\small $z_4$};
\node at (-6.9807,-1.4324) {\small $z_5$};
\node at (-4.5839,-2.7161) {\small $z_6$};

\node at (6.6311,-3.0571) {\small $z_1$};
\node at (6.8518,1.6463){\small $z_2$};
\node at (6.0395,3.4515) {\small $z_3$};
\node at (3.8331,3.8225) {\small $z_4$};
\node at (2.1484,-2.9267) {\small $z_5$};

\node at (10.1211,-2.8666) {\small $z_1$};
\node at (15.1754,-0.9411) {\small $z_2$};
\node at (18.184,-0.921) {\small $z_3$};
\node at (17.1411,4.4042) {\small $z_4$};
\node at (11.2042,4.2136) {\small $z_5$};
\node at (9.9004,0.2624) {\small $z_6$};

\node  [draw,shape=circle,outer sep=0,inner sep=1.2,minimum size=5.5,fill=black] (v100) at (-6.1758,9) {};

\node  [draw,shape=circle,outer sep=0,inner sep=1.2,minimum size=5.5,fill=black] (v101) at (0.6044,9) {};

\draw  [myedgestyle] (-5.3815,8.9932) circle (.8);

\draw  [myedgestyle] (-3.8658,8.4617) circle (.8);

\draw  [myedgestyle] (-0.9293,9.4708) circle (.8);

\node [mynodestyle] at (-4.6008,8.6964) {};
\node [mynodestyle] (v4) at (-3.1668,8.8669) {};
\node [mynodestyle] (v18) at (-2.4246,8.9972) {};
\node [mynodestyle] (v19) at (-1.6023,9.07) {};
\node [mynodestyle] (v20) at (-0.2485,9.0775) {};
\draw [myedgestyle] (v4) edge (v18);
\draw [myedgestyle] (v18) edge (v19);
\draw [myedgestyle] (v20) edge (v101);
\node [mynodestyle] at (-5.7942,8.3417) {};
\node [mynodestyle] at (-5.022,8.2715) {};
\node [mynodestyle] at (-5.724,9.6955) {};
\node [mynodestyle] at (-5.002,9.7056) {};
\node [mynodestyle] at (-4.6209,9.3546) {};
\node [mynodestyle] at (-4.0994,9.2242) {};
\node [mynodestyle] at (-3.5478,9.2097) {};
\node [mynodestyle] at (-4.4103,7.9104) {};
\node [mynodestyle] at (-3.7986,7.6698) {};
\node [mynodestyle] at (-3.1668,8.0508) {};
\node [mynodestyle] at (-0.8401,8.6826) {};
\node [mynodestyle] at (-1.7127,9.661) {};
\node [mynodestyle] at (-1.3014,10.1713) {};
\node [mynodestyle] at (-0.6295,10.197) {};
\node [mynodestyle] at (-0.1682,9.7531) {};

\node (v88) at (-6.9575,9.5178) {};
\node (v89) at (-6.9475,8.9362) {};
\node (v90) at (1.4263,9.5379) {};
\node (v91) at (1.4263,9.1568) {};
\node (v92) at (1.4664,8.6554) {};
\draw [myedgestyle] (v100) edge (v88);
\draw [myedgestyle] (v100) edge (v89);
\draw [myedgestyle] (v101) edge (v90);
\draw [myedgestyle] (v101) edge (v91);
\draw [myedgestyle] (v101) edge (v92);

\node  [draw,shape=circle,outer sep=0,inner sep=1.2,minimum size=5.5,fill=black] (v100) at (9,10) {};

\node  [draw,shape=circle,outer sep=0,inner sep=1.2,minimum size=5.5,fill=black] (v105) at (6,8.5) {};

\node  [draw,shape=circle,outer sep=0,inner sep=1.2,minimum size=5.5,fill=black] (v106) at (6,11.5) {};

\node  [draw,shape=circle,outer sep=0,inner sep=1.2,minimum size=5.5,fill=black] (v101) at (15.5,10) {};

\draw  [myedgestyle] (10.7323,10.4798) circle (.8);
\draw  [myedgestyle] (13.7349,9.944) circle (.8);

\draw [myedgestyle] (v106) edge (v100);
\draw [myedgestyle] (v105) edge (v100);
\node [mynodestyle] at (6.879,11.0721) {};
\node [mynodestyle] at (7.7433,10.6399) {};
\node [mynodestyle] at (7.0173,9.0151) {};
\node [mynodestyle] at (8.0371,9.5164) {};
\node [mynodestyle] (v93) at (10.0422,10.0177) {};
\node [mynodestyle] (v94) at (11.425,10.0349) {};
\node [mynodestyle] (v95) at (12.9641,10.0122) {};
\node [mynodestyle] (v96) at (14.5049,9.9457) {};
\draw [myedgestyle] (v100) edge (v93);
\draw [myedgestyle] (v94) edge (v95);
\draw [myedgestyle] (v96) edge (v101);
\node [mynodestyle] at (10.7164,9.6892) {};
\node [mynodestyle] at (9.9904,10.7609) {};
\node [mynodestyle] at (10.4258,11.2131) {};
\node [mynodestyle] at (11.0866,11.205) {};
\node [mynodestyle] at (11.4769,10.7436) {};
\node [mynodestyle] at (13.2399,9.3435) {};
\node [mynodestyle] at (13.8277,9.1879) {};
\node [mynodestyle] at (14.329,9.4127) {};
\node [mynodestyle] at (13.3437,10.6399) {};
\node [mynodestyle] at (14.1605,10.6062) {};
\node (v98) at (4.7442,10.7372) {};
\node (v97) at (4.5402,11.6373) {};
\node (v103) at (5.0684,7.8063) {};
\node (v102) at (4.8772,8.744) {};
\node (v99) at (4.9591,9.527) {};
\draw [myedgestyle] (v97) edge (v106);
\draw [myedgestyle] (v98) edge (v106);
\draw [myedgestyle] (v99) edge (v105);
\draw [myedgestyle] (v102) edge (v105);
\draw [myedgestyle] (v103) edge (v105);
\node (v104) at (16.5946,10.7561) {};
\node (v107) at (16.9497,9.618) {};
\draw [myedgestyle] (v104) edge (v101);
\draw [myedgestyle] (v101) edge (v107);
\end{tikzpicture}
\caption{\small Examples for Definition \ref{DEF:necklace}.}
\label{FIG: thread-necklace}
\end{center}
\end{figure}
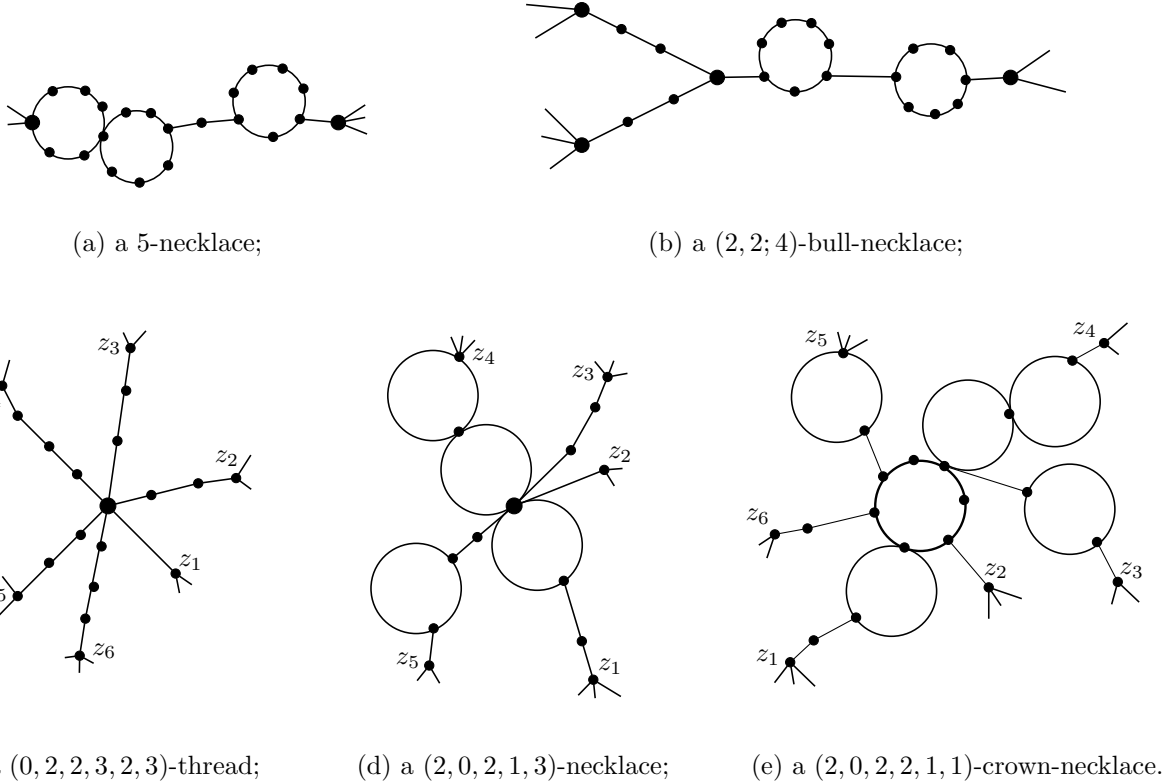

\section{The $C_p$-coloring for prime $p$}\label{sect3}
This section is aiming to show  $f(p)\le 2p^2+2p-5$ in Theorem \ref{THMmain-p}. We first present some reducible configurations under precoloring in Subsection \ref{subsectCp}, and then complete the proof in Subsection \ref{subsect-pf1.3} by a discharging method. Unlike some standard discharging arguments, our method mainly analyzes certain modified graphs obtained from the original graph, which benefits in handling some structures involving  $p$-cycles.

\subsection{Precoloring and reducible subgraphs for $C_p$-coloring}
\label{subsectCp}
Let $H$ be a thread, or a necklace, or a $(k_1,k_2,\ldots,k_t)$-thread, or a $(k_1,k_2,\ldots,k_t)$-necklace, or a $(k_1,k_2,\ldots,k_t)$-crown-necklace with $W$ being the end vertex set of $H$.
The graph $H$ is called {\em reducible} if in any graph $G$ containing $H$ as a subgraph,  any $C_p$-coloring of $G-(V(H)\setminus W)$ can be extended to a $C_p$-coloring of $G$. In other words, it is equivalent to say that $H$ is  $(\omega, W)$-colorable for any precoloring $\omega$ of $W$.
It is known from \cite{BKKW2004, PosteSR, Zhu01} that some threads and certain $(k_1,k_2,\ldots,k_t)$-threads are reducible configurations for $C_k$-coloring.
%
%
%
%
%
Our main reducible configurations in this section  are certain necklaces and crowns, generalizing from threads, for prime $p$.

We  need the following well-known Cauchy-Davenport Theorem over prime field.
For two sets $A,B$, define $A+B=\{a+b: a\in A, b\in B\}$.
\begin{theorem}\label{THM:CDthm}
  {\em (Cauchy-Davenport Theorem, \cite{Cauchy, Davenport})} Let $p$ be a prime. If $A$ and $B$ are two nonempty subset of ${\mathbb Z}_p$, then we have
  $$|A+B|\ge \min\{p, |A|+|B|-1\}.$$
\end{theorem}

\begin{lemma}\label{LEM:necklace-prime}
Let $N(x_0,x_{s+1})$ be an $s$-necklace, where for each $0\le i\le s$, there is either an edge $x_{i}x_{i+1}$ or a $p$-cycle between $x_{i}$ and $x_{i+1}$ consisting of a $k_i$-thread and a $(p-2-k_i)$-thread. Let $\omega$ be a precoloring of $x_0$, and let $B(x_i)$ be the set of available colors of $x_i$ from a coloring of $x_{i-1}$ for each $i\in [s+1]$, where $B(x_0)=\{\omega(x_0)\}$. Then each of the following holds.

(i) We have $|B(x_i)|\ge \min\{i+1, p\}$ for each $i\in [s+1]$.

 (ii) If $s\ge p-2$, then an $s$-necklace is reducible for $C_p$-coloring.
\end{lemma}
\begin{proof}
  (i) For any $s\ge i\ge 0$, we shall count the number of colors $\omega(x_{i+1})$ that can be extended from a color $\omega(x_i)$ of $x_i$.    Note that $x_0$ receives a fixed coloring $\omega(x_0)$. If $x_0x_1$ is an edge in $G$, then we have $\omega(x_1)\in\{\omega(x_0)+\frac{p-1}{2}, \omega(x_0)-\frac{p-1}{2}\}$, that is, $B(x_1)=\{\omega(x_0)+\frac{p-1}{2}, \omega(x_0)-\frac{p-1}{2}\}$. The arithmetic operations here and below are taken modulo $p$. If there is a $p$-cycle between $x_0$ and $x_1$ which consists of a $k_0$-thread and a $(p-2-k_0)$-thread, then by Lemma \ref{LEM:cycle-Ck-precoloring} we have $\omega(x_1)\in\{\omega(x_0)+\frac{p-1}{2}(k_0+1), \omega(x_0)-\frac{p-1}{2}(k_0+1)\}$, which gives $B(x_1)=\{\omega(x_0)+\frac{p-1}{2}(k_0+1), \omega(x_0)-\frac{p-1}{2}(k_0+1)\}$.  Hence, in any case, we have $|B(x_1)|=2$.

   Below we shall apply induction to show  $|B(x_i)|\ge \min\{i+1, p\}$ for each $i\in [s+1]$. The basic case $i=1$ is proved above. Assume the statement $|B(x_i)|\ge \min\{i+1, p\}$ holds for any integer at most $i$. For the case $i+1$, we shall show that $|B(x_{i+1})|\ge \min\{i+2, p\}$. Similar as before,   if $x_ix_{i+1}$ is an edge of $G$, then
   $$B(x_{i+1})=\{b+\frac{p-1}{2},b-\frac{p-1}{2}: b\in B(x_i)\};$$
  if there is a $p$-cycle between $x_i$ and $x_{i+1}$  consisting of a $k_i$-thread and a $(p-2-k_i)$-thread, then by Lemma \ref{LEM:cycle-Ck-precoloring} we have
  $$B(x_{i+1})=\{b+\frac{p-1}{2}(k_i+1),b-\frac{p-1}{2}(k_i+1): b\in B(x_i)\}.$$
  Using the notation in   Theorem \ref{THM:CDthm}, we have
  \[\text{either~} B(x_{i+1})=B(x_i)+\{\frac{p-1}{2}, -\frac{p-1}{2}\}~\text{or~} B(x_{i+1})=B(x_i)+\{\frac{p-1}{2}(k_i+1), -\frac{p-1}{2}(k_i+1)\}.
  \]
   By Theorem \ref{THM:CDthm}, we obtain that $|B(x_{i+1})|\ge \min\{|B(x_i)|+1,p\}\ge \min\{i+2, p\}$. This proves the claim that $|B(x_i)|\ge \min\{i+1, p\}$ for each $i\in[s+1]$.

   (ii) Fix a $C_p$-coloring $\omega$ of $G-(V(N(x_0,x_{s+1}))\setminus\{x_0, x_{s+1}\})$. We show that $\omega$ can be extended to a $C_p$-coloring of $G$. We still let $B(x_i)$ be the set of available colors of $x_i$ from a coloring of $x_{i-1}$ for each $1\le i \le s$, where $B(x_0)=\{\omega(x_0)\}$.  By (i), we particularly have that $|B(x_{j})|\ge p-1$ for each $p-2\le j\le s$. For the $s$-necklace $N(x_0,x_{s+1})$, $\omega(x_{s+1})$ is a fixed color, and so its restriction requires that $\omega(x_{s})\in \{\omega(x_{s+1})+\frac{p-1}{2}, \omega(x_{s+1})-\frac{p-1}{2}\}$ when $x_{s}x_{s+1}$ is an edge, and $\omega(x_{s})\in \{\omega(x_{s+1})+\frac{p-1}{2}(k_{s}+1), \omega(x_{s+1})-\frac{p-1}{2}(k_{s}+1)\}$ when there is a $p$-cycle between $x_{s}$ and $x_{s+1}$ consisting of $k_{s}$-thread and a $(p-2-k_{s})$-thread.
Since $|B(x_{s})|\ge p-1$, we have both $B(x_{s})\cap \{\omega(x_{s+1})+\frac{p-1}{2}, \omega(x_{s+1})-\frac{p-1}{2}\}\neq \emptyset$ and $B(x_{s})\cap \{\omega(x_{s+1})+\frac{p-1}{2}(k_{s}+1), \omega(x_{s+1})-\frac{p-1}{2}(k_{s}+1)\}\neq \emptyset$. Therefore, there exists an available color for the choice of $x_{s}$ in $B(x_{s})$, and so $\omega$ can be extended to a $C_p$-coloring of $G$ by appropriately coloring each of $x_1,x_2,\ldots, x_s$ and by Lemma \ref{LEM:cycle-Ck-precoloring}.
\end{proof}

\begin{lemma}\label{LEM:crown-prime}
 For a $(k_1,k_2,\ldots,k_t)$-necklace or a $(k_1,k_2,\ldots,k_t)$-crown-necklace,  if it holds that $$ \max_{1\le i\le t}\{k_i\}\le p-2 ~~\text{and}~~ \sum_{i=1}^t k_i\ge (p-2)t-p+1,$$
  then it is reducible for $C_p$-coloring.
\end{lemma}
\begin{proof}
Let $H$ be a  $(k_1,k_2,\ldots,k_t)$-necklace or a $(k_1,k_2,\ldots,k_t)$-crown-necklace with end vertex set $W$. For each $i\in[t]$, let $x_i, y_i$ be the end vertices of the $k_i$-necklace in $H$, where $x_i\in W$. If $H$ is a $(k_1,k_2,\ldots,k_t)$-necklace, then $y_1=y_2=\ldots=y_t$ is a common vertex. If $H$ is a $(k_1,k_2,\ldots,k_t)$-crown-necklace, then $y_1, y_2, \ldots, y_t$ (may or may not be identical) are lying in a common $p$-cycle. In the later case, suppose that we color a selected vertex $y_0$ of the common $p$-cycle with color $b$. Then denote the color of $y_i$ by $\varphi(y_i)=b+\frac{p-1}{2}d_i$, where $d_i$ is the distance from $y_0$ to $y_i$ in the cyclic order for each $1\le i\le t$. In the former case, we apply  the same notation and set that $y_0=y_1=y_2=\ldots=y_t$,  $\omega(y_i)=b$ and $d_i=0$ for each $i$.

Fix a precoloring $\omega$ of $W$.  We show that $\omega$ can be extended to a $C_p$-coloring of $H$ by selecting an appropriate value of $b$ with application of Lemma \ref{LEM:necklace-prime}(i).

For each $1\le i \le t$, let $B_i$ be the set of available colors of $y_i$ such that the coloring $\omega(x_i)$ and $\omega(y_i)\in B_i$ can be extended to a $C_p$-coloring of the $k_i$-necklace. By Lemma \ref{LEM:necklace-prime}(i), we have $|B_i|\ge k_i+2$. Let $D_i=\{\beta: \beta=\alpha-\frac{p-1}{2}d_i, \alpha\in B_i\}$. Clearly,  $|D_i|=|B_i|\ge k_i+2$.
%
Thus we have
\begin{eqnarray*}
  |\bigcap_{i=1}^t D_i|&\ge& \sum_{i=1}^t |D_i| - (t-1)|\bigcup_{i=1}^t D_i|\\
  &\ge& \sum_{i=1}^t (k_i+2)-(t-1)p\\
  &=& \sum_{i=1}^tk_i -(p-2)t+p\ge 1.
\end{eqnarray*}
Hence $\bigcap_{i=1}^t D_i\neq \emptyset$ holds. Then we can select an element $b\in \bigcap_{i=1}^t D_i$ and color $y_i$ with $\omega(y_i)=\varphi(y_i)=b+\frac{p-1}{2}d_i$ for each $i\in [t]$. By definition, the coloring $\omega(x_i)$ and $\omega(y_i)$ can be extended to a $C_p$-coloring of the $k_i$-necklace for each $i\in [t]$.  Therefore, $H$ is reducible for $C_p$-coloring.
\end{proof}

\subsection{The proof of Theorem~\ref{THMmain-p}}
\label{subsect-pf1.3}



By Proposition \ref{PROP:nonprimek}, $f(k)$ does not exist if $k>0$ is an odd nonprime integer. Proposition \ref{PROP:p-prime-quadratic-bound} indicates $p^2-\frac{5}{2}p+\frac{3}{2}\le f(p)$ for a prime $p\ge 5$. To complete the proof of Theorem~\ref{THMmain-p}, it suffices to show that every planar graph of girth $p$ without cycles of length from $p+1$ to $2(p-1)(p+2)-1$ is $C_p$-colorable. In fact, we show the following mild stronger theorem.

\begin{theorem}\label{THM-main1-p}
Let $G$ be a plane graph of girth $p$ without cycles of length from $p+1$ to $2(p-1)(p+2)-1$, and let $\omega$ be a precoloring of a $p$-cycle $C$ of $G$. Then $G$ is $(\omega,V(C))$-colorable.
\end{theorem}
\begin{proof}

Suppose to the contrary that $G$ is a  counterexample with  $|E(G)\setminus E(C)|$ minimized. Clearly, we have $E(G)\setminus E(C)\neq \emptyset$ and $|V(G)|>p$.

 \begin{claimab}\label{CL:degree} 
  (i) $G$ is $2$-connected. In particular, $\delta(G)\geq 2$.

  (ii) Every $p$-cycle in  $G$ bounds a face. In particular, $C$ is a facial $p$-cycle of $G$.
  \end{claimab}

\noindent {\em Proof of Claim \ref{CL:degree}.}
(i) If $G$ is not $2$-connected, then there exist proper induced subgraphs $G_1$ and $G_2$ of $G$ and a vertex
$v\in V(G_2)$ such that $E(G) = E(G_1) \cup E(G_2)$, $V(G_1) \cap V(G_2) \subseteq  \{v\}$ and $V(C)\subseteq V(G_1)$.
By the minimality of the counterexample, $\omega$ can be extended to a $C_p$-coloring $\tilde{\omega}$ of $G_1$. Take an edge $uv\in E(G_2)$, and construct a new graph $G_2'$ from $G_2$ by adding a new $(p-2)$-thread between $u$ and $v$ to form a new $p$-cycle $C'$. Then $G_2'$ contains no cycles of length from $p+1$ to $2(p-1)(p+2)-1$. Let $\omega'$ be a precoloring of $C'$ such that $\omega'(v)=\tilde{\omega}(v)$ (if $v\notin V(G_1)$, then $G$ is not connected and we  take $\omega'(v)$ to be an arbitrary color). Since $|E(G_2')\setminus E(C')|<|E(G)\setminus E(C)|$ and by the minimality of the counterexample, $\omega'$ can be extended to a $C_p$-coloring $\tilde{\omega}'$ of $G_2'$. So $\tilde{\omega}'$ and $\tilde{\omega}$ combine to provide a
$C_p$-coloring of $G$, which is a contradiction.

(ii) Suppose for a contradiction that a $p$-cycle $K$ of $G$ does not bound a face.  Let $G_1$ be the subgraph of $G$ drawn outside (and
including) $K$, and let $G_2$ be the subgraph of $G$ drawn inside (and including) $K$. We may, without loss of generality, assume that $V(C)\subset V(G_1)$.
By the minimality of the counterexample, $\omega$ can be extended to a $C_p$-coloring $\tilde{\omega}$ of $G_1$. Let $\omega'$ be the restriction of $\tilde{\omega}$ on
$V(K)$. Then $\omega'$ can be extended to a $C_p$-coloring $\tilde{\omega}'$ of $G_2$ by the minimality of $G$.
The union of $\tilde{\omega}$ and $\tilde{\omega}'$ is a $C_p$-coloring of $G$
extending $\omega$, which is a contradiction.
\q

%
%
%

~

By Claim \ref{CL:degree}(ii), $C$ must be a
facial cycle of $G$. Re-embedding $G$ on the plane if needed, we can assume that the face bounded by $C$ is the outer face
of $G$, denoted by $f_0$.
Let $H$ be a thread, or a necklace, or a $(k_1,k_2,\ldots,k_t)$-thread, or a $(k_1,k_2,\ldots,k_t)$-necklace, or a $(k_1,k_2,\ldots,k_t)$-crown-necklace of $G$ with end vertex set $W$. If $V(H)\setminus W\subseteq V(G)\setminus V(C)$, then we say $H$ is {\em valid} in $G$. 

 \begin{claimab}\label{CL:no-thread-p}
 Each of the following holds.

  (i) $G$ contains no valid $(p-2)^+$-thread.

  (ii) $G$ contains no valid $(p-2)^+$-necklace.

  (iii) $G$ contains neither a valid $(k_1,k_2,\ldots,k_t)$-necklace  nor  a valid $(k_1,k_2,\ldots,k_t)$-crown-necklace  with $\sum_{i=1}^t k_i\ge (p-2)t-p+1$, where $t\geq 3$.
  \end{claimab}

\noindent {\em Proof of Claim \ref{CL:no-thread-p}.}
Note that an $s$-thread is a special $s$-necklace without performing $C_p$-replacement operation, and so Claim \ref{CL:no-thread-p}(i) follows from Claim \ref{CL:no-thread-p}(ii).  Suppose, for a contradiction, that $G$ has a valid $(p-2)^+$-necklace, or a valid $(k_1,k_2,\ldots,k_t)$-necklace, or  a valid $(k_1,k_2,\ldots,k_t)$-crown-necklace described above, denoted by $H$ with $W$ being its end vertex set. By the minimality of $G$, $\omega$ can be extended to a $C_p$-coloring $\tilde{\omega}$ of $G-(V(H)\setminus W)$. By Lemmas \ref{LEM:necklace-prime}(ii) and  \ref{LEM:crown-prime}, $\tilde{\omega}$ can be extended to a $C_p$-coloring of $G$, that is, $\omega$ can be extended to a $C_p$-coloring of $G$,  a contradiction.
\q

~

By Claim  \ref{CL:degree}(ii),  any $p$-cycle in $G$ is a facial $p$-cycle (the boundary of a $p$-face).
Since $G$ contains no cycles of length from $p+1$ to $2(p-1)(p+2)-1$, any two $p$-cycles have no common edges.
 A vertex $v$ of a facial $p$-cycle $K$ with $d_G(v)\geq 3$  is called an {\em attachment-vertex} of $K$. Since $G$ is $2$-connected by Claim \ref{CL:degree}(i), every $p$-cycle contains at least two attachment-vertices.

Next, we construct two graphs $G'$ and $G''$ modified from $G$ for later proof.
Let $G'$ be the graph obtained from $G$ by replacing $K$ with an edge $uv$ for any facial $p$-cycle $K$ other than $C$ with exactly two attachment-vertices $u,v$.
Note that $d_G(u,v)\geq 2$ by Claim \ref{CL:no-thread-p}(i). By construction, each edge of $G'$ is corresponding to either an edge of $G$ or a $p$-cycle consisting of a $t$-thread and a $(p-2-t)$-thread with $p-2\ge t\ge 2$. Note that the shorter one of $t$-thread and  $(p-2-t)$-thread has length at most $\frac{p-1}{2}$.


Let $K$ be a facial $p$-cycle of $G'$ other than $C$ with attachment-vertices $v_1,v_2,\ldots, v_r$. Then  $r\geq 3$ by the construction of $G'$. 
To {\bf stick} $K$, we mean to delete all the vertices of $V(K)\setminus \{v_1,v_2,\ldots, v_r\}$
 and add a new vertex $v_{K}^*$ inside face $K$ to join  each vertex of
 $\{v_1,v_2,\ldots, v_r\}$. The vertex $v_K^*$ is called a {\em sticking vertex}, where the degree of $v_{K}^*$ is at least $3$.
Let $G''$ be the graph obtained from $G'$ by sticking all the facial $p$-cycles of $G'$ except $C$.

 By the construction of $G''$, we immediately  observe the following: $G''$ is a plane graph with outer face $f_0$ bounded by $C$; the minimal degree $\delta(G'')\geq 2$; for any $uv\in E(G'')$, at most one vertex of $\{u,v\}$ is  in $V(G'')\setminus V(G)$; each $2$-vertex $v$ of $G''$  is either a $2$-vertex of $G$ or an attachment-vertex of a facial $p$-cycle of $G$; each vertex $v\in V(G'')\setminus V(G)$ is a sticking vertex with $d_{G''}(v)\geq 3$. These facts will be used implicitly in the rest of the proof.

 We further obtain the claim below concerning cycles of $G''$.
 \begin{claimab}\label{CL:G''girth4p}
   The new constructed graph $G''$ is a plane graph of girth $p$ without cycles of length from $p+1$ to $4(p+2)-1$. Furthermore, $C$ is the  only one $p$-cycle of $G''$.
  \end{claimab}
  \noindent {\em Proof of Claim \ref{CL:G''girth4p}.}
  Recall that each $p$-cycle in $G$ is a facial $p$-cycle by Claim  \ref{CL:degree}(ii). By the construction of $G''$, $C$ is the only one $p$-cycle of $G''$.
  Let $Q=x_0x_1\cdots x_mx_{0}$ be a cycle of $G''$ other than $C$. If $x_i$ is a sticking vertex, then $x_i$ corresponds to a facial $p$-cycle $K_i$ of $G$, and $x_{i-1}$ and $x_{i+1}$ are two attachment-vertices of $K_i$, thus the two edges $x_{i-1}x_i, x_ix_{i+1}$ together correspond to a segment of $K_i$ whose length is at most $p-2$ as $K_i$ has at least three attachment-vertices.  If both $x_j$ and $x_{j+1}$ are not sticking vertices, then $x_jx_{j+1}$ corresponds to either an edge of $G$ or a $p$-cycle of $G$ consisting of two threads, where the shorter one has length at most $\frac{p-1}{2}$. Hence the cycle $Q$ corresponds to a cycle of length at most $\frac{p-1}{2}m$ in $G$. So we have $\frac{p-1}{2}m\ge  2(p-1)(p+2)$, which gives $m\ge 4(p+2)$.
  \q

~
%

In the graph $G''$, a thread or a $(k_1,k_2,\ldots,k_t)$-thread $H$ with end vertex set $W$  is called {\em valid} if $V(H)\setminus W\subseteq V(G'')\setminus V(C)$.
 \begin{claimab}\label{CL:thread}
  Each valid $(s+2)$-thread  of $G''$ corresponds to a valid  $s$-necklace  of $G$. In particular, $G''$ contains no valid $p^+$-thread  by Claim \ref{CL:no-thread-p}(ii).
  \end{claimab}

\noindent {\em Proof of Claim \ref{CL:thread}.} Let $P=x_0x_1\cdots x_{s+2}x_{s+3}$ be a valid $(s+2)$-thread of $G''$. For any $i\in [s+2]$, $x_i\in V(G'')\setminus V(C)$ by definition.
Noting that $x_i$ is a $2$-vertex  and by the construction of $G''$, we have $x_i\in V(G)$. Hence for each $i\in [s+1]$, the edge $x_ix_{i+1}$ in $G''$ corresponds to either an edge of $G-V(C)$ or a $p$-cycle of $G-V(C)$ consisting of two threads by its construction. Therefore, $x_1x_2\cdots x_{s+2}$ corresponds to a valid $s$-necklace  of $G$. In particular, a valid $p^+$-thread  of $G''$ corresponds to a valid $(p-2)^+$-necklace  of $G$, and thus $G''$ contains no valid $p^+$-thread by Claim \ref{CL:no-thread-p}(ii).
\q

\begin{claimab}\label{CL:t-thread}
  $G''$ contains no valid $(k_1,k_2,\ldots,k_t)$-thread  with $\sum_{i=1}^{t} k_i\ge pt-p+1$ and $t\geq 3$.
  \end{claimab}

\noindent {\em Proof of Claim \ref{CL:t-thread}.}
Suppose to the contrary that $G''$ has a valid $(k_1,k_2,\ldots,k_t)$-thread $T_x$ such that $\sum_{i=1}^{t} k_i\ge pt-p+1$ and $t=d_{G''}(x)\ge 3$. For each $i\in [t]$, let $x_i$ be the end vertex (other than $x$) of the $k_i$-thread in the $(k_1,k_2,\ldots,k_t)$-thread, let $y_i$ be the neighbor of $x_i$ on the $k_i$-thread, and let $z_i$ be the neighbor of $x$ on the $k_i$-thread. Then $V(T_x)\setminus\{y_1,y_2,\ldots,y_t\}\subseteq V(G'')\setminus V(C)$.
By Claim \ref{CL:thread}, we have $k_i\leq p-1$ for any $i\in [t]$. Note that for each $i\in [t]$, $y_i$ is not a sticking vertex since it is a $2$-vertex in $G''$. If $x$ is a vertex in $V(G)$, then the $(x,y_i)$-thread from $x$ to $y_i$ corresponds to a $(k_i-1)$-necklace $N(x,y_i)$ in $G-V(C)$ for each $i\in [t]$. Hence $G$ contains a valid $(k_1-1,k_2-1,\ldots,k_t-1)$-necklace $N_x$  with end vertices ${y_1,y_2,\ldots, y_t}$. Since $\sum_{i=1}^{t} k_i\ge pt-p+1$, we have $\sum_{i=1}^{t} (k_i-1)\ge(p-1)t-p+1$, a contradiction to Claim \ref{CL:no-thread-p}(iii).

Assume instead that $x$ is not a vertex in $V(G)$. Then $x$ is a sticking vertex in $G''$, which corresponds to a $p$-cycle $K_x$ in $G-V(C)$. Thus,  $z_i$ is an attachment-vertex of $K_x$ for each $i\in[t]$. Hence the $(z_i,y_i)$-thread  corresponds to a $(k_i-2)$-necklace $N(z_i,y_i)$ in $G-V(C)$ for each $i\in [t]$. Thus $G$ contains a valid $(k_1-2,k_2-2,\ldots,k_t-2)$-crown-necklace  with end vertices ${y_1,y_2,\ldots, y_t}$. Similarly, we have $\sum_{i=1}^{t} (k_i-2)\ge (p-2)t-p+1$ by $\sum_{i=1}^{t} k_i\ge pt-p+1$, which contradicts to Claim \ref{CL:no-thread-p}(iii) again.
\q

~

Now we shall complete the proof  by a discharging method on $G''$. Any face other than $f_0$ is called an {\em internal face} of $G''$. The vertices
of $V(G'')\setminus V(C)$ are called {\em internal vertices} of $G''$. The degree $d_{G''}(f)$ of a face $f$ is the number of
edges in its boundary, cut edges being counted twice.
Let $F(G'')$ be the set of faces of $G''$. From Euler Formula, we have
\begin{equation*}
    \sum_{v\in V(G'')}(\frac{p}{2}d_{G''}(v)-(p+2))+\sum_{f\in F(G'')}(d_{G''}(f)-(p+2))=-2(p+2),
    \end{equation*}
    which implies
    \begin{equation}\label{EQ:sum}
    \sum_{v\in V(G'')}(\frac{p}{2}d_{G''}(v)-(p+2))+(d_{G''}(f_0)+p)+\sum_{f\in F(G'')\setminus\{f_0\}}(d_{G''}(f)-(p+2)) =-2.
    \end{equation}

    Assign an initial charge $ch_0(v)=\frac{p}{2}d_{G''}(v)-(p+2)$ for each $v\in V(G'')$,  $ch_0(f_0)=2p$
    and $ch_0(f)=d_{G''}(f)-(p+2)$ for each $f\in F(G'')\setminus \{f_0\}$. Hence the total charge is $-2$ by Eq. (\ref{EQ:sum}).

~

    We redistribute the charges according to the following rules.
    {\em

    \noindent{\bf (RI)} Every $(4p+8)^+$-face of $G''$ gives charge $\frac{3}{4}$ to each of its incident internal vertices.

    \noindent{\bf (RII)} Every $3^+$-vertex of $G''$ gives charge $\frac{1}{4}$ to each of its weakly adjacent internal $2$-vertices.

    \noindent{\bf (RIII)} The outer face $f_0$ gives charge $2$ to each of its incident vertices.
    }

~

Let $ch$ denote the charge assignment after performing the charge redistribution
using the rules (RI), (RII), and (RIII).

  \begin{claimab}\label{CL:charge-face}
$ch(f)\geq 0$ for each $f\in F(G'')$.
  \end{claimab}

\noindent {\em Proof of Claim \ref{CL:charge-face}.} By Claim \ref{CL:G''girth4p}, $G''$ is a plane graph of girth $p$ without cycles of length from $p+1$ to $4(p+2)-1$. If $d_{G''}(f)=p$, then $f$ must be the outer face $f_0$, and thus $ch_0(f_0)=2p$. By (RIII), $f$ sends charge $2$ to each of its incident vertices,  and hence $ch(f)=ch_0(f_0)-2p=0$.
Now assume $d_{G''}(f)\geq 4(p+2)$.  Then $f$ sends charge $\frac{3}{4}$ to each incident internal vertices by (RI), and so $ch(f)\geq ch_0(f)-\frac{3}{4}d_{G''}(f)= (d_{G''}(f)-(p+2))-\frac{3}{4}d_{G''}(f)=\frac{1}{4}(d_{G''}(f)-4(p+2))\geq 0$. \q

 \begin{claimab}\label{CL:charge-vertex}
$ch(v)\geq 0$ for each $v\in V(G'')$.
  \end{claimab}

\noindent {\em Proof of Claim \ref{CL:charge-vertex}.}
First we assume $d_{G''}(v)=2$. Then $ch_0(v)=-2$. If $v\in V(C)$, then $v$ receives charge 2 from $f_0$ by (RIII). Thus $ch(v)=ch_0(v)+2=0$. For  an internal $2$-vertex $v$,
by Claims \ref{CL:degree} and \ref{CL:thread}, $v$ is weakly adjacent to two $3^+$-vertices, and thus $v$ receives charge $\frac{1}{4}\times 2$ by (RII).
By (RI), $v$ receives charge $\frac{3}{4}\times 2$ from its incident faces. Hence
$ch(v)=-2+\frac{1}{2}+\frac{3}{2}=0$.

Now we assume $d_{G''}(v)\geq 3$. Let $t(v)$ be the number of internal $2$-vertices weakly adjacent to $v$. By (RII), $v$ sends charge $\frac{1}{4} t(v)$ to its weakly adjacent internal $2$-vertices. If $v\in V(C)$, then $t(v)\leq (p-1)(d_{G''}(v)-2)$ as each
thread in $G''$ contains at most $(p-1)$ internal 2-vertices
by Claim \ref{CL:thread}. Note that $v$ receives charge $\frac{3}{4} (d_{G''}(v)-1)$ from its incident $(4p+8)^+$-faces by (RI), and receives charge 2 from $f_0$ by (RIII).
Then
 \begin{eqnarray*}
  ch(v)&=&ch_0(v)-\frac{1}{4} t(v)+\frac{3}{4} (d_{G''}(v)-1)+2\\
  &\geq &(\frac{p}{2}d_{G''}(v)-(p+2))-\frac{1}{4}(p-1)(d_{G''}(v)-2)+\frac{3}{4} (d_{G''}(v)-1)+2\\
  &=& \frac{p+4}{4}d_{G''}(v)-\frac{1}{2}p-\frac{5}{4}\\
  &\geq&  \frac{p+4}{4}\cdot 3-\frac{1}{2}p-\frac{5}{4}\\
  &=&\frac{p+7}{4}>0.
\end{eqnarray*}

\noindent  Assume instead that $v$ is an internal vertex.
By Claims \ref{CL:thread} and \ref{CL:t-thread},  $t(v)\leq pd_{G''}(v)-p$. By (RI), $v$ receives charge $\frac{3}{4} d_{G''}(v)$ from its incident faces.
Hence
\begin{eqnarray*}
  ch(v)&=&(\frac{p}{2}d_{G''}(v)-(p+2))+\frac{3}{4} d_{G''}(v)-\frac{1}{4} t(v)\\
  &\geq &\frac{p}{2}d_{G''}(v)-p-2+\frac{3}{4} d_{G''}(v)-\frac{1}{4} (pd_{G''}(v)-p)\\
  &=& \frac{p+3}{4}d_{G''}(v)- \frac{3}{4}p-2\\
  &\ge& \frac{p+3}{4}\cdot 3- \frac{3}{4}p-2\\
  &=&\frac{1}{4}>0.
\end{eqnarray*}
\q

By Eq. (\ref{EQ:sum}) and Claims \ref{CL:charge-face} and \ref{CL:charge-vertex},
we have
$$-2= \sum_{v\in V(G'')}ch_0(v)+\sum_{f\in F(G'')}ch_0(f)=\sum_{v\in V(G'')}ch(v)+\sum_{f\in F(G'')}ch(f)\geq 0,$$
a contradiction. This contradiction completes the proof of Theorem \ref{THM-main1-p}.
\end{proof}

\section{The fractional coloring}\label{sect4}


This section is devoted to prove Theorem \ref{THM:main-fraction-k}. We first study some  graphs with precoloring extensions  in Subsection \ref{subsect-precoloring}, serving for reducible configurations, and then present the proof of Theorem \ref{THM:main-fraction-k} in Subsection \ref{subsect-pf1.6} by a discharging method.

\subsection{Precoloring graphs for fractional $(k:\frac{k-1}{2})$-coloring}
\label{subsect-precoloring}



We start with the following property on coloring of paths.

\begin{lemma}\label{LEM:path-p}
Let $P=v_1v_2\ldots v_t$ be a path with $2\leq t\leq k$, and let $\varphi$ be a fractional $(k:\frac{k-1}{2})$-coloring of $P$. Then $|\varphi(v_1)\cap \varphi(v_t)|\geq \frac{k-t}{2}$ if $t$ is odd, and $|\varphi(v_1)\cap \varphi(v_t)|\leq \frac{t-2}{2}$ if $t$ is even.
\end{lemma}

\begin{proof}
We prove by induction. Since $\varphi$ is a fractional $(k:\frac{k-1}{2})$-coloring of $P$,
we have $\varphi(v_i)\cap \varphi(v_{i+1})=\emptyset$ for each $i\in [t-1]$. Thus Lemma \ref{LEM:path-p} holds for $t=2$.
If $t=3$, noting that $(\varphi(v_1)\cup \varphi(v_3))\subseteq [k]\setminus \varphi(v_2)$, then
$|\varphi(v_1)\cup \varphi(v_3)|\leq k-|\varphi(v_2)|$, and thus
 $|\varphi(v_1)\cap \varphi(v_3)|=|\varphi(v_1)|+|\varphi(v_3)|-|\varphi(v_1)\cup \varphi(v_3)|\ge \frac{k-1}{2}+\frac{k-1}{2}-(k-\frac{k-1}{2})=\frac{k-3}{2}$. That is, Lemma \ref{LEM:path-p} holds for $t=3$. Assume Lemma \ref{LEM:path-p} holds for any
value smaller than $t$. Now we consider $\varphi(v_1)\cap \varphi(v_t)$. First we assume $t$ is even. Then $t-1$ is odd, and $|\varphi(v_1)\cap \varphi(v_{t-1})|\geq \frac{k-t+1}{2}$ by induction hypothesis. Since $\varphi(v_t)\cap \varphi(v_{t-1})=\emptyset$, we have $\varphi(v_1)\cap \varphi(v_{t})\subseteq \varphi(v_1)\setminus \varphi(v_{t-1})$, and thus $|\varphi(v_1)\cap \varphi(v_{t})|\leq |\varphi(v_1)|-|\varphi(v_1)\cap \varphi(v_{t-1})|\leq \frac{k-1}{2}-\frac{k-t+1}{2}=\frac{t-2}{2}$. Now we assume $t$ is odd. Then $t-1$ is even, and $|\varphi(v_1)\cap \varphi(v_{t-1})|\leq \frac{t-3}{2}$ by induction hypothesis. As $\varphi(v_t)\cap \varphi(v_{t-1})=\emptyset$, we have $(\varphi(v_1)\cup \varphi(v_{t}))\subseteq [k]\setminus (\varphi(v_{t-1})\setminus\varphi(v_1))$, which implies $|\varphi(v_1)\cup \varphi(v_{t})|\leq k-|\varphi(v_{t-1})|+|\varphi(v_1)\cap \varphi(v_{t-1})|$. Thus
 $|\varphi(v_1)\cap \varphi(v_{t})|=|\varphi(v_1)|+|\varphi(v_{t})|-
|\varphi(v_1)\cup \varphi(v_{t})|\geq\frac{k-1}{2}+\frac{k-1}{2}-k+\frac{k-1}{2}-|\varphi(v_1)\cap \varphi(v_{t-1})|\geq \frac{k-1}{2}+\frac{k-1}{2}-k+\frac{k-1}{2}-\frac{t-3}{2}=\frac{k-t}{2}$.
Therefore, Lemma \ref{LEM:path-p} holds by induction.
\end{proof}

Recall that, for $S\subset V(H)$,   $H$ is {\em $\varphi_S$-colorable} if the precoloring $\varphi$ of $S$ can be extended to  a fractional $(k:\frac{k-1}{2})$-coloring of $H$.
 Note that the number $\frac{k-2}{4}+(-1)^{d(x,y)}\cdot\frac{k-2d(x,y)}{4}$ is always an integer; in fact it is $\frac{k-1-d(x,y)}{2}$ if $d(x,y)$ is even, and $\frac{d(x,y)-1}{2}$ if $d(x,y)$ is  odd.

\begin{lemma}\label{LEM:cyclep}
  Let $C$ be a cycle of length $k$. Let $\varphi$ be a precoloring of $\{x,y\}\subseteq V(C)$. Then $C$ is $\varphi_{\{x,y\}}$-colorable if and only if
  $$ |\varphi(x)\cap\varphi(y)|=\frac{k-2}{4}+(-1)^{d(x,y)}\cdot\frac{k-2d(x,y)}{4}.$$
\end{lemma}
\begin{proof}
  Denote $C=x_0x_1\ldots x_{k-1}x_0$, where $x_0=x, x_t=y$, and $d(x,y)=t\le \frac{k-1}{2}$.

Assume that $C$ is  $\varphi_{\{x,y\}}$-colorable, and let $\tilde{\varphi}$ be a fractional $(k:\frac{k-1}{2})$-coloring of $C$ extended by $\varphi$. Denote $P_1=x_0x_1\ldots x_t$ and $P_2=x_0x_{k-1}x_{k-2}\ldots x_t$. Then $P_1$ is a path of order $t+1$ and $P_2$ is a path of order $k-t+1$. Note that $\tilde{\varphi}$  also provides a fractional $(k:\frac{k-1}{2})$-coloring of $P_1$ and of $P_2$.
If $t$ is even, then by Lemma \ref{LEM:path-p}, we have
$|\tilde{\varphi}(x_0)\cap \tilde{\varphi}(x_t)|\geq \frac{k-t-1}{2}$ as $|V(P_1)|=t+1$ is odd and $|\tilde{\varphi}(x_0)\cap \tilde{\varphi}(x_t)|\leq \frac{k-t-1}{2}$ as $|V(P_2)|=k-t+1$ is even. Thus $|\varphi(x_0)\cap \varphi(x_t)|=|\tilde{\varphi}(x_0)\cap \tilde{\varphi}(x_t)|= \frac{k-t-1}{2}$.
If $t$ is odd, then by Lemma \ref{LEM:path-p}, we have
$|\tilde{\varphi}(x_0)\cap \tilde{\varphi}(x_t)|\leq \frac{t-1}{2}$ as $|V(P_1)|=t+1$ is even and $|\tilde{\varphi}(x_0)\cap \tilde{\varphi}(x_t)|\geq \frac{t-1}{2}$ as $|V(P_2)|=k-t+1$ is odd. Hence $|\varphi(x_0)\cap \varphi(x_t)|= |\tilde{\varphi}(x_0)\cap \tilde{\varphi}(x_t)|=\frac{t-1}{2}$.


  Conversely, assume that $a=|\varphi(x)\cap\varphi(y)|=\frac{k-2}{4}+(-1)^{t}\cdot\frac{k-2t}{4}.$ Without loss of generality, we may assume $\varphi(x_0)=\{1,2,\ldots, \frac{k-1}{2}\}$. If $t$ is even, we assume $\varphi(x_0)\cap \varphi(x_t)=\{1,2,\ldots, a\}$, and $\varphi(x_t)\setminus \varphi(x_0)=\{\frac{k+1}{2}+a+1,\frac{k+1}{2}+a+2,\ldots, k\}$. If $t$ is odd, we assume $\varphi(x_0)\cap \varphi(x_t)=\{\frac{k-1}{2}-a+1,\frac{k-1}{2}-a+2,\frac{k-1}{2}\}$, and $\varphi(x_t)\setminus \varphi(x_0)=\{\frac{k-1}{2}+1,\frac{k-1}{2}+2,\ldots, k-a-1\}$.
  We define a coloring by setting $\varphi(x_{2i})=\{1,2,\ldots, \frac{k-1}{2}-i\}\cup\{k-i+1,k-i+2,\ldots, k\}$ and $\varphi(x_{2i+1})=\{\frac{k+1}{2}-i,\frac{k+1}{2}-i+1,\ldots, k-i-1\}$ for $0\leq i\leq \frac{k-1}{2}$. It is routine to check that $\varphi$ is a fractional $(k:\frac{k-1}{2})$-coloring of $C$.
  \end{proof}


\begin{lemma}\label{LEM:necklace-t-small}
  Let $N(x,y)$ be a necklace with a precoloring $\varphi$ of $\{x,y\}$. Suppose that the distance between $x$ and $y$ is $d(x,y)=t\le \frac{k+1}{2}$. If $$|\varphi(x)\cap\varphi(y)|=\frac{k-2}{4}+(-1)^t\cdot\frac{k-2t}{4},$$
  then $N(x,y)$ is $\varphi_{\{x,y\}}$-colorable.
\end{lemma}
\begin{proof}
  We prove by induction. The statement holds for $t=0,1$. Assume that it hold for any value smaller than $t$. If $x$ and $y$ are in the same $k$-cycle, then the statement holds from Lemma \ref{LEM:cyclep}. Otherwise, we can always find a vertex $u$ in the shortest $(x,y)$-path $xz_1\ldots z_{t-1}y$ which divides the necklace into two separated necklaces that one is from $x$ to $u$ and the other is from $u$ to $y$. More precisely, if $xz_1$  is not contained in a $k$-cycle, then we choose $u=z_1$; otherwise, we choose $u=z_j$ where $j$ is the largest index such that $z_{j-1}z_j$ is in the $k$-cycle containing $xz_1$. Note that $u$ is a cut vertex of $H$ that divides the necklace $H$ into two separated necklaces.  Now we shall try to provide a coloring $\varphi(u)$ of $u$ and then apply induction on the $(x,u)$-necklace and on the $(u,y)$-necklace.  This can be achieved if we can find $a$ colors from $\varphi(x)\setminus \varphi(y)$, $b$ colors from $\varphi(x)\cap \varphi(y)$, $c$ colors from $\varphi(y)\setminus \varphi(x)$, and the rest colors from  $[k]\setminus(\varphi(x)\cup\varphi(y))$ to formulate $\varphi(u)$ satisfying the induction hypothesis.

  Let $d(x,u)=s$. Then $d(u,y)=t-s$. Formally, we need to find a nonnegative integer solution $(a,b,c)$ of the following system of inequalities:
  \begin{align*}
    \left\{\begin{array}{llll}
0\le a \le |\varphi(x)\setminus\varphi(y)|=\frac{k}{4}-(-1)^t\cdot\frac{k-2t}{4},\\
0\le b \le |\varphi(x)\cap\varphi(y)|=\frac{k-2}{4}+(-1)^t\cdot\frac{k-2t}{4},\\
0\le c \le |\varphi(y)\setminus\varphi(x)|=\frac{k}{4}-(-1)^t\cdot\frac{k-2t}{4},\\
0\le \frac{k-1}{2}-a-b-c \le |[k]\setminus(\varphi(x)\cup\varphi(y))|=1+\frac{k-2}{4}+(-1)^t\cdot\frac{k-2t}{4},\\
a+b=\frac{k-2}{4}+(-1)^s\cdot\frac{k-2s}{4},\\
b+c=\frac{k-2}{4}+(-1)^{t-s}\cdot\frac{k-2(t-s)}{4}.
\end{array}
\right.
  \end{align*}

Let
$$\alpha=(-1)^t\cdot\frac{k-2t}{4},~ \beta=(-1)^s\cdot\frac{k-2s}{4},~~\text{and}~~ \gamma=(-1)^{t-s}\cdot\frac{k-2(t-s)}{4}.$$

Plugging  $a=-b+\frac{k-2}{4}+\beta$ and $c=-b+\frac{k-2}{4}+\gamma$ into the above system of inequalities, we have:
\begin{align*}
    \left\{\begin{array}{llll}
\alpha+\beta-\frac{1}{2}&\le &b&\le \frac{k-2}{4}+\beta,\\
0&\le &b& \le \frac{k-2}{4}+\alpha,\\
\alpha+\gamma-\frac{1}{2}&\le &b& \le \frac{k-2}{4}+\gamma,\\
\beta+\gamma-\frac{1}{2}&\le &b& \le \frac{k}{4}+\alpha+\beta+\gamma.
\end{array}
\right.
  \end{align*}
  Let $$M=\max\{\alpha+\beta-\frac{1}{2},0,\alpha+\gamma-\frac{1}{2},\beta+\gamma-\frac{1}{2}\}~~~~\text{and}$$
  $$N=\min\{\frac{k-2}{4}+\beta,\frac{k-2}{4}+\alpha,\frac{k-2}{4}+\gamma,\frac{k}{4}+\alpha+\beta+\gamma\}.$$
We can actually show that $0\le M\le N $ by a one-by-one compression, and then setting $b=M$ provides a valid solution of the above system of inequalities. This method will be applied  in a similar but more complicated Lemma \ref{LEM:necklace-t-large} below.

Here an alternative way to do so is to check case by case on the parity as follows.
\begin{itemize}
  \item If $t$ is odd and $s$ is odd, then set $b=M=N=\frac{s-1}{2}$.
  \item If $t$ is odd and $s$ is even, then set $b=M=N=\frac{t-s-1}{2}$.
  \item If $t$ is even and $s$ is odd, then set $b=M=N=0$.
  \item If $t$ is even and $s$ is even, then set $b=M=N=\frac{k-1-t}{2}$.
\end{itemize}
Then this solution $(a,b,c)$ provides a coloring $\varphi(u)$ as desired.
\end{proof}
We present this version of the proof of Lemma \ref{LEM:necklace-t-small} to provide an overview of the more complicated Lemma \ref{LEM:necklace-t-large} below when $d(x,y)$ is relatively large.
We aslo need the following technical inequality.

\begin{proposition}\label{prop:for-MN}
  Let $s, t$ be integers with $1\le s\le \frac{k-1}{2}$ and $\frac{k+1}{2}\le t\le s+\frac{k+1}{2}$.  Denote
$$\beta=(-1)^s\cdot\frac{k-2s}{4}~~\text{and}~~ \gamma=(-1)^{t-s}\cdot\frac{k-2(t-s)}{4}.$$
Let $\ell$ be a fixed integer with $\frac{k-t}{2}-\frac{(-1)^t+1}{4}\le \ell \le\frac{t-1}{2}-\frac{(-1)^t+1}{4}.$ Define $${\bf M}=\max\{\beta+\ell-\frac{k}{4},0,\gamma+\ell-\frac{k}{4},\beta+\gamma-\frac{1}{2}\}~~~\text{and}$$
   $${\bf N}=\min\{\frac{k-2}{4}+\beta,\ell,\frac{k-2}{4}+\gamma,\beta+\gamma+\ell+\frac{1}{2}\}.$$
  Then ${\bf M}$ and ${\bf N}$ are integers satisfying
  $$0\le {\bf M} \le {\bf N}.$$
\end{proposition}

\begin{proof}  It is routine to check that each term in ${\bf M}$ and in ${\bf N}$ is an integer by discussing the parity of $t$ and $s$. To show that ${\bf M} \le {\bf N}$, is suffices to check 16 inequalities one by one.
    \begin{itemize}
      \item $\beta+\ell-\frac{k}{4} \le {\bf N}=\min\{\frac{k-2}{4}+\beta,\ell,\frac{k-2}{4}+\gamma,\beta+\gamma+\ell+\frac{1}{2}\}$;
       \item $0 \le {\bf N}=\min\{\frac{k-2}{4}+\beta,\ell,\frac{k-2}{4}+\gamma,\beta+\gamma+\ell+\frac{1}{2}\}$;
       \item  $\gamma+\ell-\frac{k}{4} \le {\bf N}=\min\{\frac{k-2}{4}+\beta,\ell,\frac{k-2}{4}+\gamma,\beta+\gamma+\ell+\frac{1}{2}\}$;
       \item  $\beta+\gamma-\frac{1}{2} \le {\bf N}=\min\{\frac{k-2}{4}+\beta,\ell,\frac{k-2}{4}+\gamma,\beta+\gamma+\ell+\frac{1}{2}\}$.
    \end{itemize}
    It turns out to become the following:
    \begin{itemize}
      \item $\ell\le \frac{k-1}{2}$, $\beta\le \frac{k}{4}$, $\beta-\gamma\le \frac{k-1}{2}-\ell$, $-\gamma\le \frac{k+2}{4}$;
       \item  $-\beta\le \frac{k-2}{4}$, $0\le \ell$, $-\gamma\le \frac{k-2}{4}$,  $-\gamma-\beta\le \ell+\frac{1}{2}$;
       \item $\gamma-\beta\le \frac{k-1}{2}-\ell$, $\gamma\le\frac{k}{4}$, $\ell\le \frac{k-1}{2}$, $-\beta\le \frac{k+2}{4}$;
       \item $\gamma\le\frac{k}{4}$, $\beta+\gamma\le \ell+\frac{1}{2}$, $\beta\le \frac{k}{4}$, $0\le \ell+1$.
               \end{itemize}
      Except some trivial ones that $|\beta|\le\frac{k}{4}$, $|\gamma|\le\frac{k}{4}$, $0\le \ell\le \frac{k-1}{2}$, this reduces to the following:

               \begin{itemize}
       \item    $-\gamma-\beta\le \ell+\frac{1}{2}$, $\beta+\gamma\le \ell+\frac{1}{2}$, $\beta-\gamma\le \frac{k-1}{2}-\ell$, and $\gamma-\beta\le \frac{k-1}{2}-\ell$.
               \end{itemize}
               Those inequalities above are all true since
                \begin{itemize}

       \item    $|\gamma+\beta|\le |\gamma|+|\beta|\le \frac{k-2s}{4}+\frac{k-2(t-s)}{4}=\frac{k-t}{2}\le \ell+\frac{1}{2}$  and
       \item $|\beta-\gamma|+\ell\le |\beta|+ |\gamma|+\ell \le \frac{k-2s}{4}+\frac{k-2(t-s)}{4}+\frac{t-1}{2}= \frac{k-1}{2}$.
               \end{itemize}
               This proves that $0\le {\bf M} \le {\bf N}.$
\end{proof}

\begin{lemma}\label{LEM:necklace-t-large}
  Let $N(x,y)$ be a necklace with a precoloring $\varphi$ of $\{x,y\}$. Suppose that the distance between $x$ and $y$ satisfies $d(x,y)=t\ge \frac{k+1}{2}$. If $$\frac{k-t}{2}-\frac{(-1)^t+1}{4}\le |\varphi(x)\cap\varphi(y)|\le \frac{t-1}{2}-\frac{(-1)^t+1}{4},$$
  then $H$ is $\varphi_{\{x,y\}}$-colorable.
\end{lemma}
\begin{proof}
  The basic case  $t=\frac{k+1}{2}$  has already been handled in Lemma \ref{LEM:necklace-t-small}. We shall prove Lemma \ref{LEM:necklace-t-large} by induction. Similarly, there exists a cut vertex $u$ of $H$ in the shortest $(x,y)$-path that divides the necklace into two parts (two separated necklaces), one is from $x$ to $u$ and the other is from $u$ to $y$. We choose such cut vertex $u$ with the smallest distance from $x$. So either $xu$ is an edge or $x$ and $u$ are in the same $k$-cycle, and hence we have $d(x,u)=s\le \frac{k-1}{2}$. We shall divide the discussion into two cases depending on the value of $d(u,y)=t-s$.

  ~

  \noindent{\bf Case 1.} $d(u,y)=t-s\le \frac{k+1}{2}.$

~~

   Note that in this case $t\le s+ \frac{k+1}{2}\le k$. Now we shall try to find $a$ colors from $\varphi(x)\setminus \varphi(y)$, $b$ colors from $\varphi(x)\cap \varphi(y)$, $c$ colors from $\varphi(y)\setminus \varphi(x)$, and the rest colors from  $[k]\setminus(\varphi(x)\cup\varphi(y))$ to formulate $\varphi(u)$ satisfying the induction hypothesis. Formally, similar as the proof of Lemma \ref{LEM:necklace-t-small}, we need to find a nonnegative integer solution $(a,b,c)$ of the following system of inequalities:
  \begin{align*}
    \left\{\begin{array}{llll}
0\le a \le |\varphi(x)\setminus\varphi(y)|,\\
0\le b \le |\varphi(x)\cap\varphi(y)|,\\
0\le c \le |\varphi(y)\setminus\varphi(x)|,\\
0\le \frac{k-1}{2}-a-b-c \le |[k]\setminus(\varphi(x)\cup\varphi(y))|,\\
a+b=\frac{k-2}{4}+(-1)^s\cdot\frac{k-2s}{4},\\
b+c=\frac{k-2}{4}+(-1)^{t-s}\cdot\frac{k-2(t-s)}{4}.
\end{array}
\right.
  \end{align*}
  Let $\ell=|\varphi(x)\cap\varphi(y)|$ be a fixed number with $\frac{k-t}{2}-\frac{(-1)^t+1}{4}\le \ell \le\frac{t-1}{2}-\frac{(-1)^t+1}{4}.$

  Denote
$$\beta=(-1)^s\cdot\frac{k-2s}{4}~~\text{and}~~ \gamma=(-1)^{t-s}\cdot\frac{k-2(t-s)}{4}.$$

  Then by plugging  $a$ and $c$ into above system of inequalities, it reduces to the following:
  \begin{align*}
    \left\{\begin{array}{llll}
\beta+\ell-\frac{k}{4}&\le &b&\le \frac{k-2}{4}+\beta,\\
0&\le &b& \le \ell,\\
\gamma+\ell-\frac{k}{4}&\le &b& \le \frac{k-2}{4}+\gamma,\\
\beta+\gamma-\frac{1}{2}&\le &b& \le \beta+\gamma+\ell+\frac{1}{2}.
\end{array}
\right.
  \end{align*}
  Let $${\bf M}=\max\{\beta+\ell-\frac{k}{4},0,\gamma+\ell-\frac{k}{4},\beta+\gamma-\frac{1}{2}\}~~~\text{and}$$
$${\bf N}=\min\{\frac{k-2}{4}+\beta,\ell,\frac{k-2}{4}+\gamma,\beta+\gamma+\ell+\frac{1}{2}\}.$$

By Proposition \ref{prop:for-MN}, ${\bf M}$ and ${\bf N}$ are integers satisfying
  $0\le {\bf M} \le {\bf N}.$ Therefore, we choose
  $$b={\bf M}, a= \frac{k-2}{4}+(-1)^s\cdot\frac{k-2s}{4}- {\bf M}, ~\text{and}~ c=\frac{k-2}{4}+(-1)^{t-s}\cdot\frac{k-2(t-s)}{4}- {\bf M},$$
  providing a desired nonnegative integer solution $(a,b,c)$.

~

  \noindent{\bf Case 2.} $d(u,y)=t-s\ge \frac{k+3}{2}.$

     ~

        We are still trying to find $a$ colors from $\varphi(x)\setminus \varphi(y)$, $b$ colors from $\varphi(x)\cap \varphi(y)$, $c$ colors from $\varphi(y)\setminus \varphi(x)$, and the rest colors from  $[k]\setminus(\varphi(x)\cup\varphi(y))$ to form $\varphi(u)$ satisfying the induction hypothesis.  This formulates similar system of inequalities as follows:
  \begin{align*}
    \left\{\begin{array}{llll}
0\le a \le |\varphi(x)\setminus\varphi(y)|,\\
0\le b \le |\varphi(x)\cap\varphi(y)|,\\
0\le c \le |\varphi(y)\setminus\varphi(x)|,\\
0\le \frac{k-1}{2}-a-b-c \le |[k]\setminus(\varphi(x)\cup\varphi(y))|,\\
a+b=\frac{k-2}{4}+(-1)^s\cdot\frac{k-2s}{4},\\
\frac{k-t+s}{2}-\frac{(-1)^{t-s}+1}{4}\le b+c\le \frac{t-s-1}{2}-\frac{(-1)^{t-s}+1}{4}.
\end{array}
\right.
  \end{align*}

  Notice that $\frac{k-2-(-1)^{\frac{k+1}{2}}}{4}$ is an integer (this value comes from the case $d(x,y)=\frac{k+1}{2}$), and we have $$\frac{k-t+s}{2}-\frac{(-1)^{t-s}+1}{4}\le \frac{k-2-(-1)^{\frac{k+1}{2}}}{4}\le \frac{t-s-1}{2}-\frac{(-1)^{t-s}+1}{4}.$$ So it is enough to find a solution $(a,b,c)$ with the last inequality replaced by $$b+c=\frac{k-2-(-1)^{\frac{k+1}{2}}}{4}.$$

  Let $\ell=|\varphi(x)\cap\varphi(y)|$ be a fixed number with $\frac{k-t}{2}-\frac{(-1)^t+1}{4}\le \ell \le\frac{t-1}{2}-\frac{(-1)^t+1}{4}.$ Note that $0\le \ell \le \frac{k-1}{2}$.

  Denote
$$\beta=(-1)^s\cdot\frac{k-2s}{4}~~\text{and}~~ \gamma=-(-1)^{\frac{k+1}{2}}\cdot\frac{1}{4}.$$

     Then with similar calculation, by plugging  $a$ and $c$ into the  above system of  inequalities, it becomes the following:
  \begin{align*}
    \left\{\begin{array}{llll}
\beta+\ell-\frac{k}{4}&\le &b&\le \frac{k-2}{4}+\beta,\\
0&\le &b& \le \ell,\\
\gamma+\ell-\frac{k}{4}&\le &b& \le \frac{k-2}{4}+\gamma,\\
\beta+\gamma-\frac{1}{2}&\le &b& \le \beta+\gamma+\ell+\frac{1}{2}.
\end{array}
\right.
  \end{align*}
  We still let $${\bf M}=\max\{\beta+\ell-\frac{k}{4},0,\gamma+\ell-\frac{k}{4},\beta+\gamma-\frac{1}{2}\}~~~\text{and}$$
   $${\bf N}=\min\{\frac{k-2}{4}+\beta,\ell,\frac{k-2}{4}+\gamma,\beta+\gamma+\ell+\frac{1}{2}\}.$$
   By Proposition \ref{prop:for-MN} with $t-s=\frac{k+1}{2}$, ${\bf M}$ and ${\bf N}$ are integers satisfying
  $0\le {\bf M} \le {\bf N}.$ Therefore, we can choose $b={\bf M}$ and corresponding $a$ and $c$ to form  a desired solution $(a,b,c)$. This completes the proof.
\end{proof}

By Lemma~\ref{LEM:necklace-t-large}, we have the following corollary.

\begin{corollary}\label{Cr-precoloring-f}
Let $N(x,y)$ be a necklace. If $d(x,y)\geq k$, then $N(x,y)$ is $\varphi_{\{x,y\}}$-colorable for
any precoloring $\varphi$ of $\{x,y\}$.
\end{corollary}

Recall Definition \ref{DEF:necklace} that a $(k_1,k_2;k_3)$-bull-necklace is a subgraph obtained from a $(k_1,k_2,k_3)$-thread by applying $C_k$-replacement operation on some edges of the $k_3$-thread. For $1\leq t\leq \frac{k-1}{2}$, let $B_v(t,s)$ be a   $(t-1,t-1;r)$-bull-necklace $N_v$ with end vertices $x,y,z$ and $d(v,z)=s$.

\begin{lemma}\label{Fr-3-necklace}
For a bull-necklace $B_v(t,s)$ with end vertices $x,y,z$, if  $1\leq t\leq \frac{k-1}{2}$ and $t+s\geq k$, then $B_v(t,s)$ is $\varphi_{\{x,y,z\}}$-colorable for any precoloring $\varphi$ of $\{x,y,z\}$ satisfying $|\varphi(x)\cap \varphi(y)|=\frac{k-1-2t}{2}$.
\end{lemma}

\begin{proof} Let $\varphi$ be a precoloring of $\{x,y,z\}$ such that $|\varphi(x)\cap \varphi(y)|=\frac{k-1-2t}{2}$.
Denote $A=\varphi(x)\setminus \varphi(y)$, $B=\varphi(x)\cap\varphi(y)$, $C=\varphi(y)\setminus \varphi(x)$, and $D=[k]\setminus (\varphi(x)\cup \varphi(y))$. Then $|A|=|C|=t$, $|B|=\frac{k-1-2t}{2}$,  and $|D|=\frac{k+1-2t}{2}$.
Let $S$ be a subset of $[k]$ such that $S=B$ if $t$ is even and $S=D$ if $t$ is odd. Then $|S|=\frac{k+1-2t}{2}-\frac{(-1)^t+1}{2}$. Denote $S_1=S\setminus \varphi(z)$ and $S_2=S\cap \varphi(z)$.

We first make the following claim.

  \begin{claimI}\label{CL:Acap} Each of the following holds:

(i) either $|A\cap \varphi(z)|\geq \frac{t+1}{2}-\frac{(-1)^t+1}{4}-|S_2|$ or $|C\cap \varphi(z)|\geq \frac{t+1}{2}-\frac{(-1)^t+1}{4}-|S_2|$;

(ii) either $|A\setminus  \varphi(z)|\geq \frac{t+1}{2}-\frac{(-1)^t+1}{4}-|S_1|$ or $|C\setminus  \varphi(z)|\geq \frac{t+1}{2}-\frac{(-1)^t+1}{4}-|S_1|$.
  \end{claimI}
\noindent {\em Proof of Claim \ref{CL:Acap}.} (i)
Notice that
\begin{eqnarray*}
|A\cap \varphi(z)|+|C\cap \varphi(z)|&=&|(A\cup C\cup S)\cap \varphi(z)|-|S_2|\\
  &= &|A\cup C\cup S|+|\varphi(z)|-|(A\cup C\cup S)\cup \varphi(z)|-|S_2|\\
  &\geq & \big[t+t+(\frac{k+1-2t}{2}-\frac{(-1)^t+1}{2})\big]+\frac{k-1}{2} -k-|S_2|\\
  &=& t-\frac{(-1)^t+1}{2}-|S_2|\\
   &\geq & 2(\frac{t+1}{2}-\frac{(-1)^t+1}{4}-|S_2|)-1.
\end{eqnarray*}
Hence (i) holds.

(ii)
Similarly, notice that
\begin{eqnarray*}
|A\setminus\varphi(z)|+|C\setminus \varphi(z)|&=&|(A\cup C\cup S)\setminus\varphi(z)|-|S_1|\\
  &\geq & (2t+\frac{k+1-2t}{2}-\frac{(-1)^t+1}{2})-\frac{k-1}{2}-|S_1|\\
  &=& t+1-\frac{(-1)^t+1}{2}-|S_1|\\
   &\geq & 2(\frac{t+1}{2}-\frac{(-1)^t+1}{4}-|S_1|).
\end{eqnarray*}
Thus (ii) holds.
\q

~

Next, we show that there are certain  subsets of $A$ and $C$ of large size for candidates of $\varphi(v)$.

  \begin{claimI}\label{CL:AC}
There exist $A_1\subseteq A\setminus \varphi(z)$, $A_2\subseteq A\cap \varphi(z)$,  $C_1\subseteq C\setminus \varphi(z)$, $C_2\subseteq C\cap \varphi(z)$ such that $|A_1|+|A_2|=\frac{t-1}{2}+\frac{(-1)^t+1}{4}$,
 $|C_1|+|C_2|=\frac{t-1}{2}+\frac{(-1)^t+1}{4}$,
$|A_1|+|S_1|+|C_1|\geq \frac{t+1}{2}-\frac{(-1)^t+1}{4}$, and
 $|A_2|+|S_2|+|C_2|\geq \frac{t+1}{2}-\frac{(-1)^t+1}{4}$.
  \end{claimI}

\noindent {\em Proof of Claim \ref{CL:AC}.} By Claim \ref{CL:Acap}(i), we may assume without loss of generality that $|A\cap \varphi(z)|\geq \frac{t+1}{2}-\frac{(-1)^t+1}{4}-|S_2|$.


 If $|C\setminus \varphi(z)|\geq  \frac{t+1}{2}-\frac{(-1)^t+1}{4}-|S_1|$, then we can choose $C_1=C\setminus \varphi(z)$ and $C_2\subseteq A\cap \varphi(z)$ such that $|C_1|+|C_2|=\frac{t-1}{2}+\frac{(-1)^t+1}{4}$ and $|C_1|\geq \frac{t+1}{2}-\frac{(-1)^t+1}{4}-|S_1|$.   This is feasible since  $|C|=t\ge \frac{t-1}{2}+\frac{(-1)^t+1}{4}$. By $|A\cap \varphi(z)|\geq \frac{t+1}{2}-\frac{(-1)^t+1}{4}-|S_2|$, we can also choose  $A_1=A\setminus \varphi(z)$ and $A_2\subseteq A\cap \varphi(z)$ such that $|A_1|+|A_2|=\frac{t-1}{2}+\frac{(-1)^t+1}{4}$ and $|A_2|\geq \frac{t+1}{2}-\frac{(-1)^t+1}{4}-|S_2|$. Hence we have $|A_1|+|S_1|+|C_1|\geq |S_1|+|C_1|\ge \frac{t+1}{2}-\frac{(-1)^t+1}{4}$ and
 $|A_2|+|S_2|+|C_2|\geq |A_2|+|S_2|\ge \frac{t+1}{2}-\frac{(-1)^t+1}{4}$.

Assume instead that $|C\setminus \varphi(z)|<\frac{t+1}{2}-\frac{(-1)^t+1}{4}-|S_1|$.
Notice that
\begin{eqnarray*}
|C\cap \varphi(z)|&=&|C|-|C\setminus\varphi(z)|\\
  &\geq & t-(\frac{t+1}{2}-\frac{(-1)^t+1}{4}-|S_1|)\\
  &=& t-\frac{t+1}{2}+\frac{(-1)^t+1}{4}+(|S|-|S_2|)\\
   &=& \frac{t-1}{2}+\frac{(-1)^t+1}{4}+(\frac{k+1-2t}{2}-\frac{(-1)^t+1}{2})-|S_2|\\
   &=& \frac{k-t}{2}-\frac{(-1)^t+1}{4}-|S_2|\\
   &\geq & \frac{t+1}{2}-\frac{(-1)^t+1}{4}-|S_2|.
\end{eqnarray*}
Hence we can choose $C_1=C\setminus \varphi(z)$ and $C_2\subseteq C\cap \varphi(z)$ such that $|C_1|+|C_2|=\frac{t-1}{2}+\frac{(-1)^t+1}{4}$ and $|C_2|\geq \frac{t+1}{2}-\frac{(-1)^t+1}{4}-|S_2|$. By Claim \ref{CL:Acap}(ii) and as $|C\setminus \varphi(z)|<\frac{t+1}{2}-\frac{(-1)^t+1}{4}-|S_1|$,  we have $|A\setminus \varphi(z)|\geq \frac{t+1}{2}-\frac{(-1)^t+1}{4}-|S_1|$. Thus we can select $A_1=A\setminus \varphi(z)$ and $A_2\subseteq A\cap \varphi(z)$ such that $|A_1|+|A_2|=\frac{t-1}{2}+\frac{(-1)^t+1}{4}$ and $|A_1|\geq \frac{t+1}{2}-\frac{(-1)^t+1}{4}-|S_1|$.
Therefore, we have $|A_1|+|S_1|+|C_1|\geq |A_1|+|S_1|\ge \frac{t+1}{2}-\frac{(-1)^t+1}{4}$ and
 $|A_2|+|S_2|+|C_2|\geq |S_2|+|C_2|\ge \frac{t+1}{2}-\frac{(-1)^t+1}{4}$ as desired.\q

~

%
%

~
Now we choose such $A_1\subseteq A\setminus \varphi(z)$, $A_2\subseteq A\cap \varphi(z)$,  $C_1\subseteq C\setminus \varphi(z)$, and $C_2\subseteq C\cap \varphi(z)$
as in Claim \ref{CL:AC}.
Let $\varphi(v)=A_1\cup A_2\cup S_1\cup S_2\cup C_1\cup C_2$. Then $$|\varphi(v)|=(\frac{t-1}{2}+\frac{(-1)^t+1}{4})+(\frac{k+1-2t}{2}-\frac{(-1)^t+1}{2})
+(\frac{t-1}{2}+\frac{(-1)^t+1}{4})=\frac{k-1}{2}.$$  Moreover, $|\varphi(v)\cap \varphi(x)|=|A_1|+|A_2|+|S_1| +|S_2|=\frac{k-1-t}{2}$ if $t$ is even and $|\varphi(v)\cap \varphi(x)|=|A_1|+|A_2|=\frac{t-1}{2}$ if $t$ is odd; $|\varphi(v)\cap \varphi(y)|=|C_1|+|C_2|+|S_1| +|S_2|=\frac{k-1-t}{2}$ if $t$ is even and $|\varphi(v)\cap \varphi(x)|=|C_1|+|C_2|=\frac{t-1}{2}$ if $t$ is odd.

Notice that $(A_1\cup S_1\cup C_1)\subset [k]\setminus \varphi(z)$ and $(A_2\cup S_2\cup C_2)\subset \varphi(z)$. Hence by Claim  \ref{CL:AC} we have
$$\frac{t+1}{2}-\frac{(-1)^t+1}{4}\leq |\varphi(v)\cap \varphi(z)|\leq \frac{k-1}{2}-(\frac{t+1}{2}-\frac{(-1)^t+1}{4}).$$ Since $s+t\geq k$, we have
\[\frac{k-s}{2}-\frac{(-1)^s+1}{4}\le \frac{t+1}{2}-\frac{(-1)^t+1}{4} ~~\text{and}~~
\frac{k-1}{2}-(\frac{t+1}{2}-\frac{(-1)^t+1}{4})\leq \frac{s-1}{2}-\frac{(-1)^s+1}{4},
\] which implies
$$\frac{k-s}{2}-\frac{(-1)^s+1}{4}\leq |\varphi(v)\cap \varphi(z)|\leq \frac{s-1}{2}-\frac{(-1)^s+1}{4}.$$ Thus $B_v(t,s)$ is $\varphi_{\{x,y,z,v\}}$-colorable by Lemmas~~\ref{LEM:necklace-t-small} and \ref{LEM:necklace-t-large}.
\end{proof}

%
%
%
%
%
%
%

\subsection{The proof of Theorem~\ref{THM:main-fraction-k}}
\label{subsect-pf1.6}

Now we are ready to prove Theorem~\ref{THM:main-fraction-k} restated below in terms of plane graph.

~

\noindent{\bf Theorem~\ref{THM:main-fraction-k}.} {\em
 For any odd integer $k\ge 5$, every plane graph of girth at least $k$ without cycles of length from $k+1$ to $\lfloor\frac{22}{3}k\rfloor$  is fractional  $(k:\frac{k-1}{2})$-colorable.
}

\begin{proof}
Suppose, for a contradiction, that $G$ is a counterexample with $|V(G)|+|E(G)|$ minimized.

 \begin{claim}\label{CL:2-conected-f}
  $G$ is $2$-connected. In particular, $\delta(G)\geq 2$.
  \end{claim}

\noindent {\em Proof of Claim \ref{CL:2-conected-f}.}
Clearly, $G$ is connected. If $G$ is not $2$-connected, then there exist proper induced subgraphs $G_1$ and $G_2$ of $G$ and a vertex
$v\in V(G_2)$ such that $E(G) = E(G_1) \cup E(G_2)$ and $V(G_1) \cap V(G_2) = \{v\}$. By the minimality of $G$, $G_1$ has a fractional  $(k:\frac{k-1}{2})$-coloring $\varphi_1$ and $G_2$ has a fractional  $(k:\frac{k-1}{2})$-coloring $\varphi_2$. Exchange the colors if needed  such that $\varphi_1(v) =\varphi_2(v)$, then $\varphi_1$ and $\varphi_2$ combine to become a
$(k:\frac{k-1}{2})$-coloring of $G$, which is a contradiction.
\q

%

For $1\leq t\leq \frac{k-1}{2}$, let $F_v(t,s)$ be a graph obtained from a bull-necklace $B_v(t,s)$ with end vertices $x,y,z$ by joining a new $(x,y)$-path of length $k-2t$ connecting $x$ and $y$, where the vertices in the new $(x,y)$-path may have arbitrary degrees in $G$. That is,  $F_v(t,s)$ consists of a $k$-cycle $C_v$ and a necklace $N(v,z)$ with a common vertex $v$, where in the $k$-cycle $C_v$ there exist two $(t-1)$-threads, one is from $x$ to $v$ and the other is from $y$ to $v$.

 \begin{claim}\label{CL:F(t-s)}
 $G$ contains no $F_v(t,s)$ with $1\leq t\leq \frac{k-1}{2}$ and $t+s\geq k$, where $s=d(v,z)$.
  \end{claim}

\noindent {\em Proof of Claim \ref{CL:F(t-s)}.} Suppose to the contrary that $G$ contains an $F_v(t,s)$ with $1\leq t\leq \frac{k-1}{2}$ and $t+s\geq k$. By the minimality of $G$, $G-(V(N(v,z))\setminus \{v,z\})$ has a $(k:\frac{k-1}{2})$-coloring $\varphi$.  If $2t\leq \frac{k-1}{2}$, then $d_G(x,y)=2t$, and if $2t\geq \frac{k+1}{2}$, then $d_G(x,y)=k-2t$.
By Lemma \ref{LEM:cyclep}, we always have $|\varphi(x)\cap\varphi(y)|=\frac{k-1-2t}{2}$. Let $\varphi'$ be the restriction of $\varphi$ to $G-(V(B_v(t,s))\setminus \{x,y,z\})$. As $|\varphi'(x)\cap\varphi'(y)|=\frac{k-1-2t}{2}$,    $B_v(t,s)$ is $\varphi'_{\{x,y,z\}}$-colorable by Lemma \ref{Fr-3-necklace}. That is, $\varphi'$ can be extended to a $(k:\frac{k-1}{2})$-coloring of $G$, which is a contradiction. \q

%

~~

From $G$, we obtain a subgraph $G'$ as follows: {\em for each facial $k$-cycle $C$ of $G$, if there exists a  $2$-vertex in $C$, then we
delete all the $2$-vertices of a longest thread of $C$.}
Clearly, the obtained graph $G'$ is a plane graph of girth at least $k$, and contains no cycles of length from $k+1$ to $\lfloor\frac{22k}{3}\rfloor$; furthermore, each facial $k$-cycle of $G'$ contains no $2$-vertices.  It is easy to see that $G'$ has minimal degree at least $2$ by its construction.

Let $T(v,x)$ be a $(v,x)$-thread of $G'$ and let $u=N_{G'}(v)\cap V(T(v,x))$. If there exists $w\in N_{G'}(v)\setminus \{u\}$ such that $vu$ and $vw$ are in a common $k$-cycle of $G$, then we say $v$ is a {\em bad end vertex} of $T(v,x)$; otherwise,  $v$ is called a {\em good end vertex} of $T(v,x)$.

  \begin{claim}\label{CL:nothread-good-end}
Let $T(v,x)$ be a $(v,x)$-thread of $G'$ with  a good end vertex  $v$. Then $d_{G'}(v,x)\leq k-1$.
  \end{claim}

\noindent {\em Proof of Claim \ref{CL:nothread-good-end}.}   Suppose to the contrary that $d_{G'}(v,x)\geq k$. If $x$ is also a good end vertex of $T(v,x)$, then the thread $T(v,x)$ in $G'$ corresponds to a necklace $H$ with end vertices ${v,x}$ in $G$.
By the minimality of $G$, $G-(V(H)\setminus \{v,x\})$ has a $(k:\frac{k-1}{2})$-coloring $\varphi$. Since $d_G(v,x)=d_{G'}(v,x)\geq k$ by construction, $H$ is $\varphi_{\{v,x\}}$-colorable by Corollary \ref{Cr-precoloring-f}. That is, $\varphi$ can be extended to a $(k:\frac{k-1}{2})$-coloring of $G$, which is a contradiction.

 Therefore we assume that $x$ is a bad end vertex of $T(v,x)$. By definition, let $y=N_{G'}(x)\cap V(T(v,x))$ such that there exists a $k$-cycle $C_x$ of $G$ containing both $xy$ and $xz$ for some
 $z\in N_{G'}(x)\setminus \{y\}$.
Let $w\in V(C_x)\cap V(T(v,x))$ such that $d_G(x,w)$ as large as possible.
By the construction of $G'$, we obtain that the $(x,w)$-thread  in $G$ satisfies $d(x,w)\leq \frac{k-1}{2}$, and that there is a deleted thread from $w$ to some vertex, say $(w,u)$-thread, in the $k$-cycle $C_x$ such that $d(u,w)\ge d(x,w)$. Thus
$G$ contains a bull-necklace $B_w(d(w,x),d(w,v))$, which provides an $F_w(d(w,x),d(w,v))$ in $G$,  contradicting to Claim \ref{CL:F(t-s)}.
\q

  \begin{claim}\label{CL:nothread-f}
  $G'$ contains no  $(\frac{3k-3}{2})^+$-thread.
  \end{claim}

\noindent {\em Proof of Claim \ref{CL:nothread-f}.}
Suppose to the contrary that $G'$ has a $(\frac{3k-3}{2})^+$-thread $T(v,x)$. Then $d_{G}(v,x)=d_{G'}(v,x)\geq \frac{3k-1}{2}$. By Claim \ref{CL:nothread-good-end},  $v$ and $x$ are both bad end vertices of $T(v,x)$. Let $u$ be the neighbor of $v$ in $T(v,x)$. Then there exists a $k$-cycle $C_v$ of $G$ containing both $vu$ and $vw$ for some
 $w\in N_{G'}(v)\setminus \{u\}$.
Let $y\in V(C_v)\cap V(T(v,x))$ such that $d_G(v,y)$ is as large as possible.
By the construction of $G'$, we have $d(v,y)\leq \frac{k-1}{2}$, and so $d(x,y)=d(v,x)-d(v,y) \geq k$. Now $T(v,x)-(V(T(v,x))\cap V(C_v)\setminus \{y\})$ is an $(x,y)$-thread from $x$ to $y$ in $G'$  with $y$ being a good end vertex, which is a contradiction to Claim \ref{CL:nothread-good-end}.
\q

%
%
%

 \begin{claim}\label{CL:no-3-thread-f}
  $G'$ contains no  $(k_1,k_2,k_3)$-thread such that $k_1+k_2+k_3\geq \frac{11k-17}{3}$.
  \end{claim}

\noindent {\em Proof of Claim \ref{CL:no-3-thread-f}.}
Suppose to the contrary that $G'$ has a $(k_1,k_2,k_3)$-vertex $v$ such that $k_1+k_2+k_3\geq \frac{11k-17}{3}$ with end vertices ${x,y,z}$. Then $d_{G'}(v,x)+d_{G'}(v,y)+d_{G'}(v,z)\geq \frac{11k-8}{3}$.

If there are no two edges incident to $v$ in $G'$ lying in a common $k$-cycle of $G$, then  we may assume, without loss of generality, that $d_{G'}(v,x)\ge\frac{1}{3}(d_{G'}(v,x)+d_{G'}(v,y)+d_{G'}(v,z))>k$. Now the $(x,v)$-thread  from $x$ to $v$ has length at least $k$  with $v$ as a good end vertex, which is a contradiction to Claim \ref{CL:nothread-good-end}.


If there exist two edges incident to $v$ in $G'$ containing in a $k$-cycle $C_v$ of $G$, then we may suppose that $C_v$ has no common vertex other than $v$ with the $(v,z)$-thread $T(v,z)$.
Thus $v$ is a good end vertex of the $(v,z)$-thread $T(v,z)$, and so $d_{G'}(v,z)\le k-1$ by Claim \ref{CL:nothread-good-end}. Let $u$ be the common vertex of $C_v$ and the  $(v,x)$-thread $T(v,x)$ such that $d_G(v,u)$ as large as possible, and let $w$ be the common vertex of $C_v$ and the $(v,y)$-thread $T(v,y)$ such that $d_G(v,w)$ as large as possible. By the construction of $G'$, we have $d_{G'}(v,u)+d_{G'}(v,w)\leq \frac{2k}{3}$, since the deleted $(u,w)$-thread is a longest thread in $C_v$.  Now we have
\begin{eqnarray*}
  d_{G'}(x,u)+d_{G'}(y,w)&=& d_{G'}(v,x)+d_{G'}(v,y)+d_{G'}(v,z) - (d_{G'}(v,u)+d_{G'}(v,w))-d_{G'}(v,z)\\
  &\ge& \frac{11k-8}{3}- \frac{2k}{3} - (k-1) =2k-\frac{5}{3}.
\end{eqnarray*}
Thus $\max\{d_{G'}(x,u), d_{G'}(y,w)\}\ge k$, say $d_{G'}(x,u)\geq k$. Hence the $(x,u)$-thread $T(x,u)$ is of length at least $k$ with $u$ as a good end vertex, a contradiction to Claim \ref{CL:nothread-good-end}. \q

%
%
%
%

~

Now we  complete the proof  by a discharging method on $G'$.

Let $F(G')$ be the set of faces of $G'$. From Euler Formula, we have
    \begin{equation}\label{EQ:sum-f-k}
    \sum_{v\in V(G')}(\frac{k-2}{2}d_{G'}(v)-k)+\sum_{f\in F(G')}(d_{G'}(f)-k)=-2k.
    \end{equation}

    Assign an initial charge $ch_0(v)=\frac{k-2}{2}d_{G'}(v)-k$ for each $v\in V(G')$, and $ch_0(f)=d_{G'}(f)-k$ for each $f\in F(G')$. Hence the total charge is $-2k$ by the Eq. (\ref{EQ:sum-f-k}).

    ~

    We redistribute the charges according to the following rules.

   {\em \noindent{\bf (R1)} Every $\lceil\frac{22k}{3}\rceil^+$-face  of $G'$ gives charge $\frac{19}{22}$ to each of its incident vertices.

    \noindent{\bf (R2)} Every $3^+$-vertex  of $G'$ gives charge $\frac{3}{22}$ to each of its weakly adjacent $2$-vertices.
    }

~

Let $ch$ denote the charge assignment after performing the charge redistribution
using  rules (R1) and (R2).

  \begin{claim}\label{CL:charge-face-f}
$ch(f)\geq 0$ for $f\in F(G')$.
  \end{claim}

\noindent {\em Proof of Claim \ref{CL:charge-face-f}.}
Clearly, each $k$-face $f$ has charge $ch(f)=ch_0(f)= 0$. Each $\lceil\frac{22k}{3}\rceil^+$-face $f$ sends charge $\frac{19}{22}$ to each incident vertices by (R1). So $ch(f)=ch_0(f)-\frac{19}{22}d_{G'}(f)=(d_{G'}(f)-k)-\frac{19}{22}d_{G'}(f)=\frac{3}{22}d_{G'}(f)-k\geq 0$ as $d_{G'}(f)\geq \lceil\frac{22k}{3}\rceil$.
\q

 \begin{claim}\label{CL:charge-vertex-f}
$ch(v)\geq 0$ for $v\in V(G')$.
  \end{claim}

\noindent {\em Proof of Claim \ref{CL:charge-vertex-f}.}
Let $v$ be a vertex of $G'$. Then $d_{G'}(v)\geq 2$ by Claim \ref{CL:2-conected-f} and the construction of $G'$.

First we assume $d_{G'}(v)=2$. Then $ch_0(v)=-2$. By Claims \ref{CL:2-conected-f} and \ref{CL:nothread-f}, $v$ is weakly adjacent to two $3^+$-vertex, and thus $v$ receives charge $\frac{3}{22}\times 2$ by (R2).
By (R1), $v$ receives charge $\frac{19}{22}\times 2$ from its two incident faces. Hence
$ch(v)=-2+\frac{3}{22}\times 2+\frac{19}{22}\times 2=0$.

Now we assume $d_{G'}(v)\geq 3$. Let $t(v)$ be the number of $2$-vertices weakly adjacent to $v$. Suppose $v$ is adjacent to $r(v)$ facial $k$-cycles.  Since $G'$ contains no cycles of length from $k+1$ to $\lfloor\frac{22k}{3}\rfloor$, any two $k$-cycles of $G'$ have no edges in common, and thus $r(v)\leq \frac{d_{G'}(v)}{2}$. By Claim \ref{CL:nothread-f} and by the construction of $G'$, each thread incident to $v$ contains at most $(\frac{3k-3}{2}-1)$ $2$-vertices and each $k$-cycle contains no $2$-vertices, and so we have $t(v)\leq \frac{3k-5}{2}(d_{G'}(v)-2r(v))$. By (R1), $v$ receives charge $\frac{19}{22}(d_{G'}(v)-r(v))$ from its incident faces. By (R2), $v$ sends $3/22$ to each of its weakly adjacent $2$-vertices. Therefore, we have
\begin{equation}\label{EQ:chv322}
  ch(v)=(\frac{k-2}{2}d_{G'}(v)-k)+\frac{19}{22}(d_{G'}(v)-r(v))-\frac{3}{22}t(v).
\end{equation}

Assume that $d_{G'}(v)\geq 4$. By Eq. (\ref{EQ:chv322}), it follows from $t(v)\leq \frac{3k-5}{2}(d_{G'}(v)-2r(v))$  that
\begin{eqnarray*}
  ch(v) &\geq &\frac{k-2}{2}d_{G'}(v)-k+\frac{19}{22}(d_{G'}(v)-r(v))-\frac{3}{22}\times\frac{3k-5}{2}(d_{G'}(v)-2r(v))\\
  &=& \frac{13k+9}{44} d_{G'}(v)-k+\frac{9k-34}{22}r(v)\\
    &\geq & \frac{13k+9}{44} d_{G'}(v)-k\\
    &\geq & \frac{13k+9}{44} \cdot 4-k\\
     &=& \frac{2k+9}{11}>0.
\end{eqnarray*}

Assume instead that $d_{G'}(v)=3$. Then $ch_0(v)=\frac{k-6}{2}$ and $r(v)\leq 1$. If $r(v)=1$, then $t(v)\leq \frac{3k-5}{2}$ by Claim \ref{CL:nothread-f}. Thus by Eq.~(\ref{EQ:chv322}) we have  $ch(v)=\frac{k-6}{2}+\frac{19}{22}\times 2-\frac{3}{22}\times \frac{3k-5}{2}=\frac{13k-41}{44}\geq \frac{6}{11}$.
If $r(v)=0$, then $t(v)\leq \frac{1}{3}(11k-17)$ by Claim \ref{CL:no-3-thread-f}. Thus by Eq.~(\ref{EQ:chv322}) we have $ch(v)=\frac{k-6}{2}+\frac{19}{22}\times 3-\frac{3}{22}\times \frac{11k-17}{3}=\frac{4}{11}>0$. This proves Claim \ref{CL:charge-vertex-f}.

\q

Combining Eq. (\ref{EQ:sum-f-k}), Claims \ref{CL:charge-face-f} and \ref{CL:charge-vertex-f},
we have
$$-2k= \sum_{v\in V(G')}ch_0(v)+\sum_{f\in F(G')}ch_0(f)=\sum_{v\in V(G')}ch(v)+\sum_{f\in F(G')}ch(f)\geq 0,$$
a contradiction. This contradiction finishes the proof of Theorem \ref{THM:main-fraction-k}.
\end{proof}

\section{Concluding Remarks}\label{sect5}
In this paper, we obtain two Steinberg-type results on circular coloring and fractional coloring as Theorems \ref{THMmain-p} and \ref{THM:main-fraction-k}. Improving the bound to $f(p)\le p(p-2)$ would provide solutions to Conjecture \ref{CONJ:circular-conj} for $t=p-1$ when $p\ge 5$ is a prime, and completely determining the value $f(p)$ seems to be  more challenging.   Theorem \ref{THM:main-fraction-k} confirms the fractional coloring version of Conjecture \ref{CONJ:f(p)} for $p\ge 11$,  since $\frac{22p}{3}\le p(p-2)$ when $p\ge 11$. In a followup paper \cite{HL20}, we also verify the remaining cases ($p=5,7$) of the fractional coloring version of Conjecture \ref{CONJ:f(p)} with refined arguments and additional configurations. Those results provide evidences to Conjectures \ref{CONJ:circular-conj} and \ref{CONJ:f(p)}.

A nature question is to consider  variations of Question \ref{QEST:f(k)} concerning odd cycles. However,  naive odd cycle versions of Theorems \ref{THMmain-p} and \ref{THM:main-fraction-k} are false, i.e., for any $t>\frac{k-1}{2}$, there exist  planar graphs $G$ of odd girth $k$ without odd cycles of length from $k+2$ to $2t+1$ satisfying $\chi_c(G)\geq \chi_f(G)>\frac{2k}{k-1}$. To see this, we construct a graph $G$ by taking $2t$ disjoint copies of $k$-cycle, where each $k$-cycle contains two distinguished edges $x_iy_i,y_iz_i$ for each $i\in[2t]$, adding edges $x_iy_{i+1}, z_iy_{i+1}$ for each $i\in[2t-1]$, and adding a new vertex $v$ to connect edges $vx_{2t},vz_{2t},vy_1$. See Figure \ref{Fig:forbidodd} for the construction of $G$. We claim that $\chi_f(G)>\frac{2k}{k-1}$. In fact, if $\varphi$ is a fractional $(ka,\frac{k-1}{2}a)$-coloring of $G$, then it is easy to show, by an  argument similar to Lemmas \ref{LEM:path-p} and \ref{LEM:cyclep}, that  $|\varphi(x_i)\cap \varphi(z_i)|=\frac{k-1}{2}a-a$. This implies $\varphi(y_i)=\varphi(y_{i+1})$ for each $i\in[2t-1]$ and $\varphi(y_{2t})=\varphi(v)$, which indicates $\varphi(y_1)=\varphi(v)$. But there is an edge $y_1v$ between $y_1$ and $v$, a contradiction. Hence $\chi_c(G)\geq \chi_f(G)>\frac{2k}{k-1}$.

\begin{figure}[h]
\begin{center}
  \begin{tikzpicture}

\tikzstyle{mynodestyle} = [draw,shape=circle,outer sep=0,inner sep=1.2,minimum size=3.5,fill=black]
\tikzstyle{myedgestyle} = [line width=0.6pt, black]

\node  [draw,shape=circle,outer sep=0,inner sep=1.2,minimum size=4.5,fill=black] (v2) at (-6,2) {};

\node  [draw,shape=circle,outer sep=0,inner sep=1.2,minimum size=3.5,fill=black] (v21) at (-4,2) {};

\node  [draw,shape=circle,outer sep=0,inner sep=1.2,minimum size=3.5,fill=black] (v22) at (-2,2) {};

\node  [draw,shape=circle,outer sep=0,inner sep=1.2,minimum size=3.5,fill=black] (v23) at (0,2) {};

\node  [draw,shape=circle,outer sep=0,inner sep=1.2,minimum size=3.5,fill=black] (v24) at (2,2) {};

\node  [draw,shape=circle,outer sep=0,inner sep=1.2,minimum size=3.5,fill=black] (v28) at (4,2) {};
\node  [draw,shape=circle,outer sep=0,inner sep=1.2,minimum size=4.5,fill=black] (v25) at (6,2) {};

\node [mynodestyle] (v1) at (-5,3.5) {};
\node [mynodestyle] (v3) at (-5,2.9) {};
\node [mynodestyle] (v4) at (-5,2.3) {};
\node [mynodestyle] (v5) at (-5,1.7) {};
\node [mynodestyle] (v6) at (-5,1.1) {};
\node [mynodestyle] (v7) at (-5,0.5) {};

\node [mynodestyle] (v8) at (-3,3.5) {};
\node [mynodestyle] at (-3,2.9) {};
\node [mynodestyle] at (-3,2.3) {};
\node [mynodestyle] at (-3,1.7) {};
\node [mynodestyle] at (-3,1.1) {};
\node [mynodestyle] (v9) at (-3,0.5) {};

\node [mynodestyle] (v10) at (-1,3.5) {};
\node [mynodestyle] at (-1,2.9) {};
\node [mynodestyle] at (-1,2.3) {};
\node [mynodestyle] at (-1,1.7) {};
\node [mynodestyle] at (-1,1.1) {};
\node [mynodestyle] (v11) at (-1,0.5) {};

\node [mynodestyle] (v12) at (1,3.5) {};
\node [mynodestyle] at (1,2.9) {};
\node [mynodestyle] at (1,2.3) {};
\node [mynodestyle] at (1,1.7) {};
\node [mynodestyle] at (1,1.1) {};
\node [mynodestyle] (v13) at (1,0.5) {};

\node [mynodestyle] (v14) at (3,3.5) {};
\node [mynodestyle] at (3,2.9) {};
\node [mynodestyle] at (3,2.3) {};
\node [mynodestyle] at (3,1.7) {};
\node [mynodestyle] at (3,1.1) {};
\node [mynodestyle] (v15) at (3,0.5) {};

\node [mynodestyle] (v16) at (5,3.5) {};
\node [mynodestyle]  at (5,2.9) {};
\node [mynodestyle]  at (5,2.3) {};
\node [mynodestyle]  at (5,1.7) {};
\node [mynodestyle]  at (5,1.1) {};
\node [mynodestyle] (v17) at (5,0.5) {};

\draw [myedgestyle] (v1) edge (v3);
\draw [myedgestyle] (v4) edge (v3);
\draw [myedgestyle] (v4) edge (v5);
\draw [myedgestyle] (v6) edge (v7);
\draw [myedgestyle] (v5) edge (v6);
\draw [myedgestyle] (v8) edge (v9);
\draw [myedgestyle] (v10) edge (v11);
\draw [myedgestyle] (v12) edge (v13);
\draw [myedgestyle] (v14) edge (v15);

\draw [myedgestyle] (v2) edge (v1);
\draw [myedgestyle] (v2) edge (v7);
\draw [myedgestyle] (v1) edge (v21);
\draw [myedgestyle] (v7) edge (v21);
\draw [myedgestyle] (v21) edge (v8);
\draw [myedgestyle] (v21) edge (v9);
\draw [myedgestyle] (v9) edge (v22);
\draw [myedgestyle] (v22) edge (v8);
\draw [myedgestyle] (v22) edge (v10);
\draw [myedgestyle] (v22) edge (v11);
\draw [myedgestyle] (v11) edge (v23);
\draw [myedgestyle] (v23) edge (v10);
\draw [myedgestyle] (v23) edge (v12);
\draw [myedgestyle] (v23) edge (v13);
\draw [myedgestyle] (v13) edge (v24);
\draw [myedgestyle] (v24) edge (v12);
\draw [myedgestyle] (v24) edge (v14);
\draw [myedgestyle] (v24) edge (v15);

\draw [myedgestyle] plot[smooth, tension=.7] coordinates {(v2)};
\draw [myedgestyle] plot[smooth, tension=.7] coordinates {(v2)};
\draw [myedgestyle] plot[smooth, tension=.7] coordinates {(v2) (-5.3,4) (5.3,4) (v25)};

\node at (-6.1,1.7) {\small $y_1$};
\node at (-4,1.62) {\small $y_2$};
\node at (-2,1.62) {\small $y_3$};
\node at (0,1.62) {\small $y_4$};
\node at (2,1.62) {\small $y_5$};
\node at (4,1.62) {\small $y_6$};
\node at (6.1,1.7) {\small $v$};

\node at (-5,3.75)  {\small $x_1$};
\node at (-5,0.25) {\small $z_1$};
\node at (-3,3.75)  {\small $x_2$};
\node at (-3,0.25) {\small $z_2$};
\node at (-1,3.75)  {\small $x_3$};
\node at (-1,0.25) {\small $z_3$};
\node at (1,3.75)  {\small $x_4$};
\node at (1,0.25) {\small $z_4$};
\node at (3,3.75)  {\small $x_5$};
\node at (3,0.25) {\small $z_5$};
\node at (5,3.75)  {\small $x_6$};
\node at (5,0.25) {\small $z_6$};

\draw [myedgestyle] (v14) edge (v28);
\draw [myedgestyle] (v28) edge (v15);
\draw [myedgestyle] (v28) edge (v16);
\draw [myedgestyle] (v28) edge (v17);
\draw [myedgestyle] (v16) edge (v25);
\draw [myedgestyle] (v25) edge (v17);
\draw [myedgestyle] (v16) edge (v17);
\end{tikzpicture}
\end{center}
\caption{Construction of $G$ when $t=3$ and $k=7$.}\label{Fig:forbidodd}
\end{figure}
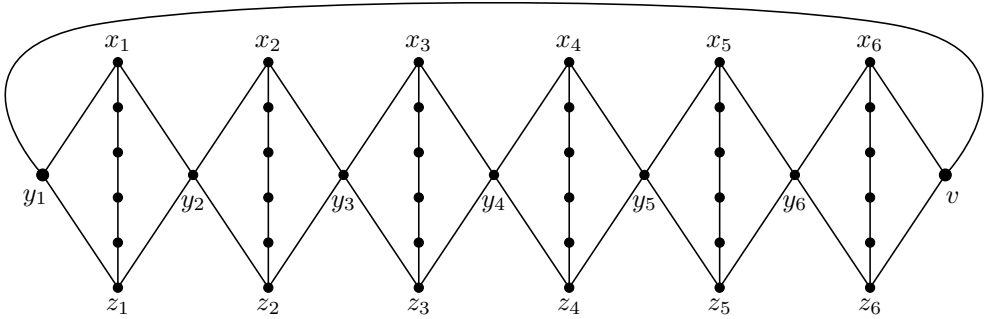

 Two $k$-cycles are called {\em adjacent} if they share at least one common edge. Notice that the above constructed graph $G$ contains adjacent $k$-cycles. It would be possible to consider the following modified odd cycle versions without adjacent  $k$-cycles.

\begin{question}
  Does there exist a smallest number $g(p)$ for each prime $p\ge 3$ such that every planar graph of odd girth $p$ without adjacent $p$-cycles and without odd cycles of length from $p+2$ to $g(p)$ is $C_{p}$-colorable?
  \end{question}
  The results from \cite{BGMR09,CHKLS17, Xu06} imply that $g(3)=7$.   It would be interesting to show the existence of $g(p)$ for every prime $p\ge 5$. Furthermore, is it true that $g(p)\le f(p)+1$?

  A similar question arises for fractional coloring.
  \begin{question}
  Does there exist a smallest number $h(k)$ for each odd integer $k\ge 3$ such that every planar graph of odd girth $k$ without adjacent $k$-cycles and without odd cycles of length from $k+2$ to $h(k)$ is fractional $(k:\frac{k-1}{2})$-colorable?
  \end{question}
  From Theorem \ref{THM:main-fraction-k}, it is plausible that $h(k)$ exists as a linear function of $k$.

\end{document}